\pdfoutput=1

\documentclass[english]{jnsao}
\usepackage{amsthm}
\usepackage[nameinlink,capitalize]{cleveref}

\usepackage{enumitem}
\setlist[enumerate,1]{label={(\roman*)}}
\setlist[enumerate,2]{label={(\alph*)}}

\theoremstyle{definition}
\newtheorem{assumption}[theorem]{Assumption}

\crefname{assumption}{Assumption}{Assumptions}
\crefname{remark}{Remark}{Remarks}
\crefname{example}{Example}{Examples}
\crefname{lemma}{Lemma}{Lemmas}
\crefname{theorem}{Theorem}{Theorems}
\crefname{proposition}{Proposition}{Propositions}
\crefname{corollary}{Corollary}{Corollaries}
\crefname{definition}{Definition}{Definitions}

\renewcommand{\epsilon}{\varepsilon}
\renewcommand{\phi}{\varphi}

\newcommand{\N}{\mathbb{N}}
\newcommand{\R}{{\mathbb{R}}}

\newcommand{\E}{\mathbb{E}}
\newcommand{\di}{\mathrm{d}}
\newcommand{\Borel}{\mathcal{B}}
\newcommand{\Normal}{\mathcal{N}}
\newcommand{\Laplace}{\mathcal{L}}

\newcommand{\umap}{\hat{u}_\mathrm{MAP}}

\newcommand{\Prob}[1]{\mathbb{P}\left[#1\right]}
\newcommand{\Exp}[1]{\mathbb{E}\left[#1\right]}
\newcommand{\X}[1]{{\mathcal{X}^{#1}}}

\DeclareMathOperator{\dom}{\mathcal{D}}
\DeclareMathOperator{\ran}{\mathcal{R}}
\DeclareMathOperator{\Tr}{tr}

\DeclareMathOperator*{\argmin}{arg\,min}

\DeclareMathOperator{\Var}{Var}

\DeclareMathOperator{\supp}{supp}

\newcommand{\abs}[1]{{\lvert#1\rvert}}
\newcommand{\bigabs}[1]{\left\lvert#1\right\rvert}
\newcommand{\norm}[1]{{\lVert#1\rVert}}
\newcommand{\bignorm}[1]{\left\lVert#1\right\rVert}
\newcommand{\scalprod}[1]{{(#1)}}
\newcommand{\bigscalprod}[1]{\left(#1\right)}
\newcommand{\dualpair}[1]{{\langle#1\rangle}}
\newcommand{\qm}[1]{``#1''} 

\title{Are minimizers of the Onsager--Machlup functional strong posterior modes?}
\shorttitle{Variational characterization of posterior modes}
\author{Remo Kretschmann\thanks{Institute of Mathematics, University of Würzburg (\email{remo.kretschmann@mathematik.uni-wuerzburg.de}).}}
\shortauthor{Remo Kretschmann}
\date{2023-07-03}
\hypersetup{
    pdftitle = {Are minimizers of the Onsager--Machlup functional strong posterior modes?},
    pdfauthor = {Remo Kretschmann},
    pdfkeywords = {maximum a posteriori estimator, inverse problems, Bayesian inference, nonparametric statistics, asymptotic statistics, impulsive noise, Tikhonov regularization, inverse heat equation},
}

\manuscriptcopyright{{\copyright} Remo Kretschmann}
\manuscriptlicense{}

\begin{document}

\maketitle

\begin{abstract}
In this work we connect two notions: That of the nonparametric mode of a probability measure, defined by asymptotic small ball probabilities, and that of the Onsager--Machlup functional, a generalized density also defined via asymptotic small ball probabilities. We show that in a separable Hilbert space setting and under mild conditions on the likelihood, modes of a Bayesian posterior distribution based upon a Gaussian prior exist and agree with the minimizers of its Onsager--Machlup functional and thus also with weak posterior modes.
We apply this result to inverse problems and derive conditions on the forward mapping under which this variational characterization of posterior modes holds. Our results show rigorously that in the limit case of infinite-dimensional data corrupted by additive Gaussian or Laplacian noise, nonparametric maximum a posteriori estimation is equivalent to Tikhonov--Phillips regularization. In comparison with the work of Dashti, Law, Stuart, and Voss (2013), the assumptions on the likelihood are relaxed so that they cover in particular the important case of white Gaussian process noise.
We illustrate our results by applying them to a severely ill-posed linear problem with Laplacian noise, where we express the maximum a posteriori estimator analytically and study its rate of convergence in the small noise limit.
\end{abstract}

\section{Introduction}

Maximum a posteriori (MAP) estimates are a useful and widely used way to describe points of maximal probability under a Bayesian posterior distribution. Unlike the conditional mean, their computation does not require numerical integration with respect to the posterior distribution, which may be prohibitively expensive for certain problems. Instead, they can typically be found as the solution to an optimization problem, which in many situations allows their efficient computation.

MAP estimates refer to modes of a Bayesian posterior distribution. The modes of a probability measure on $\R^n$ with a continuous Lebesgue density are simply the maximizers of this density. Probability measures on infinite-dimensional spaces do not have a Lebesgue density, which makes it necessary to generalize the definition of a mode to cover this setting. A common way to do this is to consider a point as a mode if small balls around it have asymptotically maximal probability as made rigorous in \cref{def_mode}. A problem with this definition is that it is not immediately clear how to find or compute such points.
One way to generalize the notion of a (negative logarithmic) Lebesgue density is the Onsager--Machlup functional, which describes the asymptotic relative probability of small balls around points under the posterior distribution as made rigorous in \cref{def_omf}. For Bayesian posterior distributions, it typically has the form of a penalized negative log-likelihood, and since its minimizers describe points of maximal probability, they are natural candidates to be modes. The question under which conditions minimizers of the Onsager--Machlup functional coincide with MAP estimates in the context of nonparamatric Bayesian inference is, however, a matter of ongoing research.

A first attempt at answering this question in a separable Banach space setting and for Gaussian priors has been made in \cite{Dashti:2013}. However, the proof therein is incomplete. On the one hand, parts of the proof in general only hold for separable Hilbert spaces. On the other hand, there remain several gaps in the proof in the Hilbert space case. This will be discussed in more detail in \cref{sect_main_res}. In the author's PhD thesis \cite{Kretschmann:2019}, most of these gaps were closed. In \cite{klebanov2023}, a remaining issue was resolved, completing the proof in the Hilbert space case.

The coincidence of MAP estimators and minimizers of the Onsager--Machlup functional has been proven for Besov-$1$ priors \cite{agapiou2018}, and, recently, for diagonal Gaussian priors on $\ell^p$ \cite{klebanov2023}. The question whether this coincidence also holds in the general separable Banach space case has very recently been answered positively in \cite{Lambley:2023}.

Helin and Burger introduced the more general notion of a weak mode (or weak MAP estimate in their terminology) in \cite{helin2015}, which is closely related to the Onsager--Machlup functional. Ayanbayev, Klebanov, Lie, and Sullivan further generalized this notion in \cite{Ayanbayev_2021a}, see \cref{def_weak_mode}, and showed that, under certain conditions, weak modes can be characterized as minimizers of the Onsager--Machlup functional. This will be discussed in more detail in \cref{sect_modes_OMF}. Since any mode is also a weak mode, this immediately yields a variational characterization of the mode for unimodal distributions. Lie and Sullivan studied the coincidence of weak and strong modes in \cite{Lie_Sul_2018}, i.e., the question whether every weak mode is also a mode. They found that under a uniformity condition, this is indeed the case, but that there are also examples of weak modes that are not modes. Said uniformity condition is, however, rather restrictive, as it assumes that all balls around a weak mode with small enough radius already have maximal probabilty among balls of the same radius. This excludes for example skewed distributions, and its practical relevance is limited. Beyond this result, it not known under which conditions weak and strong modes coincide.
In \cite{Ayanbayev_2021a,Ayanbayev_2021b}, the convergence of weak modes is connected to the $\Gamma$-convergence of Onsager--Machlup functionals, covering Bayesian inverse problems with Gaussian and Besov priors and countable product measures on weighted $\ell^p$ spaces.

We are particularly interested in Bayesian inverse problems with additive noise. Here, the Onsager--Machlup functional has the form of a Tikhonov--Phillips functional, so that its minimizers can be understood as regularized solutions. Previous results regarding the coincidence of MAP estimates and minimizers of the Onsager--Machlup functional are insofar unsatisfactory as that they do not cover the important limit case of inverse problems with infinite-dimensional data corrupted by Gaussian process noise. 
Both in \cite{Dashti:2013} and \cite{klebanov2023}, the likelihood is required to be uniformly bounded from above. Likelihoods arising from inverse problems with infinite-dimensional data are, however, typically not bounded from above both in case of Gaussian and Laplacian noise, see \cref{ex:gauss_unbounded,ex:unbounded_phi}. In \cite{agapiou2018}, the scope was limited to inverse problems with finite data. In \cite{Kretschmann:2019}, the log-likelihood was only required to be Lipschitz continuous, thereby allowing for infinite-dimensional Laplacian noise, but not for infinite-dimensional Gaussian noise.

While in reality, all data is finite, it is of interest to consider the theoretical limit case of infinite-dimensional data, which can be considered as measuring with \qm{infinite resolution}. This has the advantage that both theoretical results and numerical methods do not depend on the resolution of the measurement and are ensured not to break down as the resolution increases, cf.~\cite{lassas2009}. This limit case is not to be confused with that of a finite-dimensional measurement (corresponding to a finite resolution) that is repeated infinitely many times.
The setting of white Gaussian process noise and infinite-dimensional data is widely used in statistical and Bayesian inverse problems, see, e.g., \cite{Lehtinen1989,lassas2009,ding2017,bur_hel_kek_2018,aga_mat_2018,Nickl2020Bernstein,giordano2020bernstein,kretschmann2022}.
In many applications, a Gaussian distribution of the noise arises either naturally or is justified as an approximation by the central limit theorem if the exact distribution of the noise is non-Gaussian or unknown and the number of measurements (with finite resolution) is high enough.

Impulsive noise such as salt-and-pepper noise or noise with random valued outliers arises in many applications in image and signal processing, e.g., image acquisition with faulty pixels in a sensor or faulty memory locations \cite{bov_2005}. Such impulsive noise functions take very large values on a small part of their domain, while taking small values or being indentical to zero elsewhere.
Tikhonov--Phillips regularization with an $L^1$-data fitting term is more robust toward outliers than $L^2$-data fitting and has been observed to provide better estimates for inverse problems with impulsive noise \cite{kaer_kun_maj_2005,cla_jin_kun_2010}. Laplacian countable product noise is of particular interest for a rigorous Bayesian interpretation of such regularization methods because it yields a negative log-likelihood that has the form of a weighted $\ell^1$-norm, as we will see in \cref{sec:laplacian_noise}.

In \cite{Cavalier_2008}, an overview over optimal convergence rates of the minimax risk in the small noise limit is given for mildly and severely ill-posed linear inverse problems under Sobolev-type and analytic source conditions.
In \cite{aga_mat_2018}, the convergence rate of the MAP estimator in the linear conjugate-Gaussian setting is studied in these cases, both with and without a data-dependent choice of the prior mean.
For nonlinear inverse problems with finite-dimensional data perturbed by Gaussian noise, almost sure weak convergence of the MAP estimator in the small noise limit has been established in section 4.2 of \cite{Dashti:2013}.

\subsection{Contributions}

In this work, we relax the assumptions on the likelihood in such a way that they allow for the case of Gaussian process noise. We assume that the negative log-likelihood is Lipschitz continuous on bounded sets and satisfies a lower cone condition, see \cref{ass:phi}. 
We show that under these assumptions, MAP estimates do in fact coincide with minimizers of the Onsager--Machlup functional for general Bayesian posterior distributions based upon Gaussian priors on a separable Hilbert space. This provides a mean to compute MAP estimates both analytically and numerically. Conversely, it guarantees that penalized maximum likelihood estimates indeed describe points of maximal posterior probability. Since weak MAP estimates are known to correspond to minimizers of the Onsager--Machlup functional in the considered setting, our result moreover shows that here, weak and strong MAP estimates coincide.

We then establish Lipschitz conditions on the forward mapping for nonlinear Bayesian inverse problems with infinite-dimensional data under which aforementioned assumptions on the likelihood are satisfied. We do this for two types of additive noise: Gaussian noise, which covers in particular the important case of white Gaussian noise, and Laplacian noise. Our results show rigorously that Bayesian MAP estimation is equivalent to Tikhonov--Phillips regularization in the considered cases.

Eventually, we apply our results to a class of severely ill-posed linear inverse problems that include, i.a., the inverse heat equation. Here, we express the MAP estimator explicitly and study its rate of convergence in case of Laplacian noise.

\subsection{Organization of this paper}

This work is structured as follows. 
In \cref{sect_var_char}, we introduce and discuss the definitions of strong modes, weak modes, and the Onsager--Machlup functional in terms of small ball probabilities, we state and discuss our assumptions on the likelihood, we formulate our main result --- a variational characterization of posterior modes in case of a Gaussian prior --- and carry out its proof.
In \cref{sect_ip}, we derive the likelihood for Bayesian inverse problems in case of infinite-dimensional additive Gaussian and Laplacian noise, and state conditions on the forward mapping under which a variational characterization of MAP estimates is possible.
In \cref{sect_lin_prob}, we consider a severely ill-posed linear problem, derive the posterior distribution and the MAP estimate in case of Laplacian noise, and study the rate of convergence of the MAP estimator in the small noise limit.
\Cref{sect_proofs_var_char,sect_proofs_ip,sect_proofs_lin_prob} contain technical proofs from \cref{sect_var_char,sect_ip,sect_lin_prob}.


\section{Variational characterization of maximum a posteriori estimates}
\label{sect_var_char}

The central result in this section is the coincidence of MAP estimates with minimizers of the Onsager--Machlup functional under mild assumptions on the likelihood.

\subsection{Modes and the Onsager--Machlup functional}
\label{sect_modes_OMF}

MAP estimates are widely used in Bayesian inverse problems because they can typically be found as the solution to an optimization problem and do not require integration with respect to the posterior distribution. Their statistical interpretation in a nonparametric setting may, however, not always be straightforward.
If a measure on $\R^n$ has a continuous Lebesgue density, then its modes are simply the maximizers of this density. 
Although every numerical computation of a mode is finite dimensional, it is desirable to have a well-defined notion of a mode in the infinite-dimensional limit case of which the mode in the finite-dimensional case can be considered an approximation.
For probability measures on infinite-dimensional spaces, which cannot have a Lebesgue density, modes are usually defined in terms of small ball probabilities.
Let us consider a probability measure $\mu$ on the Borel $\sigma$-algebra $\Borel(X)$ of a separable Banach space $X$.
The following definition has been introduced in \cite[Def.~3.1]{Dashti:2013}.

\begin{definition}
	\label{def_mode}
	A point $\hat{x} \in X$ is called a \emph{mode} (or \emph{strong mode}) of $\mu$, if it satisfies
	\begin{equation}
		\label{eq:def_mode}
		\lim_{\varepsilon \to 0} \frac{\mu(B_\varepsilon(\hat{x}))}{\sup_{x \in X} \mu(B_\varepsilon(x))} = 1,
	\end{equation}
	where $B_\varepsilon(x) \subset X$ denotes the open ball with radius $\varepsilon$ centred at $x \in X$.
\end{definition}

Weaker definitions of a mode have been introduced in \cite[Def.~4]{helin2015} and \cite[Def.~2.3]{Lie_Sul_2018}, and generalized in \cite[Def.~3.7]{Ayanbayev_2021a} as follows.

\begin{definition}
	\label{def_weak_mode}
	A point $\hat{x} \in \supp \mu \subseteq X$ is called a \emph{weak mode} of $\mu$, if it satisfies
	\[ \limsup_{\epsilon \to 0} \frac{\mu(B_\epsilon(x))}{\mu(B_\epsilon(\hat{x}))} \le 1 \quad \text{for all}~x \in X, \]
	where
	\[ \supp(\mu) := \left\{ x \in X: \mu(B_\epsilon(x)) > 0~\text{for all}~\epsilon > 0 \right\} \]
	denotes the topological support of $\mu$.	
\end{definition}

Both definitions can be generalized to topological vector spaces using bounded, open neighborhoods instead of norm balls, see \cite{Lie_Sul_2018}. Note that even simple probability measures on $\R$ with a Lebesgue density may fail to have a mode both in the strong and the weak sense, see Example 2.2 in \cite{CHKP_2018}.
Since the ratio in \eqref{eq:def_mode} is always smaller or equal to $1$, the mode property is equivalent to
\[ \limsup_{\epsilon \to 0} \frac{\sup_{x \in X} \mu(B_\epsilon(x))}{\mu(B_\epsilon(\hat{x}))} \le 1. \]
The difference between the two definitions is thus the order in which the supremum and the limit superior are formed. The probability of small balls around a weak mode asymptotically dominates the probability of balls of the same radius around any other fixed center point. In contrast, the probability of small balls around a strong mode is asymptotically equal to the maximal achievable probability of balls of the same radius. Here, the supremum is taken for every radius individually. If the supremum is attained, it may be attained by different center points for different radii.

By definition, every mode is also a weak mode, see also Lemma 3.9 in \cite{Ayanbayev_2021a}, but not every weak mode is a strong mode, see Example B.5 in \cite{Ayanbayev_2021a}. It is generally an open question if or under which conditions weak and strong modes coincide. It was shown in \cite[Thm.~2.6]{Lie_Sul_2018} that under the uniformity condition
\[ \mu(B_\epsilon(\hat{x})) = \sup_{x \in X} \mu(B_\epsilon(x)) \]
for all small enough $\epsilon$, a weak mode $\hat{x}$ is also a strong mode. This condition is, however, quite restrictive --- it excludes for example skewed probability measures such as a $\Gamma$-distribution with shape and rate parameters $\alpha = \beta = 2$ --- and is thus of limited use in practice.
As pointed out in \cite[Lem.~2.2]{Lambley:2023}, if the measure $\mu$ has a strong mode, then all its weak modes are also strong modes. This shifts the question of coincidence of weak and strong modes to the question of existence of a strong mode.

Another approach to describing points that in a certain sense maximize the probability under a measure is seeking minimizers of its Onsager--Machlup functional, which plays the role of a generalized negative log-density for measures without a Lebesgue density. In our context, it is defined as follows.
Let us again consider a probability measure $\mu$ on the Borel $\sigma$-algebra of a separable Banach space $X$.

\begin{definition}
	\label{def_omf}
	Let $E \subset \supp(\mu)$ denote the set of admissible shifts for $\mu$ that yield an equivalent measure, i.e., all $h \in X$ for which the shifted measure
	\begin{equation*}
		\mu_h := \mu(\cdot - h)
	\end{equation*}
	is equivalent with $\mu$.
	A functional $I$: $E \to \R$ is called an \emph{Onsager--Machlup functional} of $\mu$, if
	\begin{equation}
		\label{def_OM_funct}
		\lim_{\varepsilon \to 0} \frac{\mu(B_\varepsilon(h_1))}{\mu(B_\varepsilon(h_2))} = \exp\left(I(h_2) - I(h_1)\right) \quad\text{for all}~h_1,h_2 \in E.
	\end{equation}
\end{definition}

The Onsager--Machlup functional describes the asymptotic ratio of small ball probabilities around any two points in $E$. Although it is only unique up to the addition of a constant, we will speak of \emph{the} Onsager--Machlup functional since all representatives have the same extremal points. For points outside of $E$, the limit on the left hand side of \eqref{def_OM_funct} does not need to exist and can be infinite. As a first example, we consider the Onsager--Machlup functional of a Gaussian measure. Let $\ran(F)$ denote the range of a mapping $F$.

\begin{proposition}
	\label{OM_funct_gauss}
	Let $\Normal_Q$ be a nondegenerate centered Gaussian measure with covariance operator $Q$ on a separable Hilbert space $X$.
	Then the set of admissible shifts for $\Normal_Q$ is given by its Cameron--Martin space $E = \ran(Q^{1/2})$, and its Onsager--Machlup functional $I$: $E \to \R$ by
	\[ I(x) = \frac12\norm{x}_Q^2 \quad \text{for all}~x \in E, \]
	where $\norm{\cdot}_Q := \norm{Q^{-1/2}\cdot}_X$ denotes the Cameron--Martin norm of $\Normal_Q$.
\end{proposition}
\begin{proof}
By the Cameron--Martin theorem \cite[Thm.~2.8]{DaPrato:2006}, the space $E$ of admissible shifts for the measure $\Normal_Q$ is given by $\ran(Q^{1/2})$.
It follows from Proposition 3 in Section 18 of \cite{Lifshits:1995} that
\begin{align*}
	\lim_{r \to 0} \frac{\Normal_Q(B_r(h_1))}{\Normal_Q(B_r(h_2))}
	&= \lim_{r \to 0} \frac{\Normal_Q(rB_1(0))}{\Normal_Q(h_2 + rB_1(0))} \cdot \lim_{r \to 0} \frac{\Normal_Q(h_1 + rB_1(0))}{\Normal_Q(rB_1(0))} \\
	&= \exp \left(\frac12\big\|Q^{-\frac12}h_2\big\|_X^2\right) \exp \left(-\frac12\big\|Q^{-\frac12}h_1\big\|_X^2\right)
\end{align*}
for all $h_1,h_2 \in E$.
\end{proof}

In nonparametric Bayesian inference, it is not immediately clear when strong modes of the posterior distribution coincide with the minimizers of its Onsager--Machlup functional.
It has been proven in \cite[Prop.~4.1]{Ayanbayev_2021a} that weak modes agree with the minimizers of the Onsager--Machlup functional if $\mu$ and $E$ have the so-called $M$-property, which holds if there exists an $x_0 \in E$ such that
\[ \lim_{\epsilon \to 0} \frac{\mu(B_\epsilon(x))}{\mu(B_\epsilon(x_0))} = 0 \quad \text{for all}~x \in X \setminus E. \]
Here, $E$ can be an arbitrary subset of $X$.
It has been shown in \cite[Lem.~4.5]{klebanov2023} that Gaussian measures on $\ell^p$, $1 \le p < \infty$, with diagonal covariance have the $M$-property in combination with their Cameron--Martin space. 
Very recently, the $M$-property has been established for centered nondegenerate Gaussian measures on separable Banach spaces in combination with their Cameron--Martin space, see \cite[Cor.~3.3]{Lambley:2023}.

\subsection{Bayesian set-up}

Now, we consider a Bayesian posterior distribution $\mu^y$ on the Borel $\sigma$-algebra $\Borel(X)$ of a separable Hilbert space $X$ whose density with respect to a Gaussian prior distribution is given by Bayes' formula.
The data $y$ may be an element of another separable Hilbert space $Y$ but is considered as fixed throughout this section, i.e., we consider the posterior distribution $\mu^y$ inferred from one specific realization of the data.

\begin{assumption}
	\label{ass:post_dist}
	\begin{enumerate}
		\item \label{ass:prior_gauss} The parameter $u$ has a nondegenerate centered Gaussian prior distribution $\mu_0 := \Normal_\Sigma$ on $X$, i.e., the covariance operator $\Sigma$ is injective.
		\item \label{ass:bayes_formula} The posterior distribution $\mu^y$ is absolutely continuous with respect to $\mu_0$, and there exist a measurable function $\Phi$: $X \to \R$ and $Z > 0$ such that
		\begin{equation}
			\label{post_dist}
			\frac{\di\mu^y}{\di\mu_0}(u) = \frac{\exp(-\Phi(u))}{Z}.
		\end{equation}
	\end{enumerate}
\end{assumption}

Note that the covariance operator of any Borel probability measure on a separable Hilbert space is self-adjoint, positive, and trace class, see \cite[Prop.~1.8]{DaPrato:2006}.
\Cref{ass:post_dist} (ii) is a standard assumption in Bayesian inverse problems, see, e.g., section 2.4 and Theorem 6.29 in \cite{stuart_2010} or Theorem 14 in \cite{das_stu_2017}.
We make the following assumptions on the negative log-likelihood $\Phi$.

\begin{assumption}
	\label{ass:phi}
	The function $\Phi$: $X \to \R$ satisfies the following two conditions.
	\begin{enumerate}
		\item \label{ass:phi_loc_lip} $\Phi$ is \emph{Lipschitz continuous on bounded sets}, i.e., for every $r > 0$, there exists $L = L(r) > 0$ such that
		\[ \abs{\Phi(x_1) - \Phi(x_2)} \le L \norm{x_1 - x_2}_X \quad \text{for all}~x_1, x_2 \in B_r(0). \]
		\item \label{ass:phi_lower_bound} There exists $\underline{L} \ge 0$ such that $\Phi$ satisfies the \emph{lower cone condition}
		\[ \Phi(x) \ge \Phi(0) - \underline{L}\norm{x}_X \quad \text{for all}~x \in X. \]
	\end{enumerate}
\end{assumption}

\Cref{ass:phi} is, in particular, satisfied if $\Phi$ is Lipschitz continuous, i.e., if there exists $L > 0$ such that
\[ \abs{\Phi(x_1) - \Phi(x_2)} \le L\norm{x_1 - x_2}_X \quad \text{for all}~x_1, x_2 \in X. \]
A Lipschitz continuous log-likelihood arises for example in inverse problems with Laplacian noise, see \cref{sec:laplacian_noise}.
We will moreover see in \cref{sec:gaussian_noise} that \cref{ass:phi} covers the important case of infinite-dimensional Gaussian noise because the lower cone condition in $0$ allows for positive quadratic growth of $\Phi$. This case is neither covered by the assumptions used in \cite{Dashti:2013} and \cite{klebanov2023}, which include a uniform lower bound on $\Phi$, see \cref{ex:gauss_unbounded}, nor by those used in \cite{Kretschmann:2019}, which demand global Lipschitz continuity.

The Onsager--Machlup functional of $\mu^y$ is now given by the sum of the Onsager--Machlup functional of the prior distribution and the negative log-likelihood $\Phi$. Moreover, it possesses a minimizer.
\begin{proposition}
	\label{form_OM_funct}
	Let \cref{ass:post_dist} hold.
	If $\Phi$ satisfies \cref{ass:phi}, then the set of admissible shifts of $\mu^y$ is given by $E = \ran(\Sigma^{1/2})$ and its Onsager--Machlup functional $I$: $E \to \R$ has the form
	\begin{equation}
		\label{def_I}
		I(x) = \Phi(x) + \frac12\norm{x}_\Sigma^2 \quad\text{for all}~x \in E.
	\end{equation}
\end{proposition}

\begin{proposition}
	\label{min_I_ex}
	If $\Phi$ satisfies \cref{ass:phi}, then $I$, as given by \eqref{def_I}, has a minimizer in $E$.
\end{proposition}

The proofs of these two propositions can be found in \cref{sect_proofs_var_char}.

\subsection{Main result}
\label{sect_main_res}

The main result of this work is the following.
\begin{theorem} 
	\label{var_char_modes}
	Under \cref{ass:post_dist,ass:phi}, the following holds true.
	\begin{enumerate}
		\item The posterior distribution $\mu^y$ has a mode.
		\item A point $x \in X$ is a mode of $\mu^y$ if and only if it minimizes the Onsager--Machlup functional of $\mu^y$.
	\end{enumerate}
\end{theorem}

Note that any minimizer of the Onsager--Machlup functional lies in $E$ by definition.
\Cref{var_char_modes} gives a positive answer to the aforementioned question in case of a Gaussian prior distribution on a separable Hilbert space: Under mild conditions on the likelihood, nonparametric posterior modes do indeed agree with minimizers of the Onsager--Machlup functional. This opens up the possibility of computing posterior modes explicitly by solving a canonical optimization problem and gives a statistical interpretation of the chosen objective functional.

In \cref{form_OM_funct}, we have seen that the Onsager--Machlup functional of $\mu^y$ has the form of a Tikhonov--Phillips functional with discrepancy term $\Phi(x)$ and quadratic penalty term $\frac12\norm{x}_\Sigma^2$, which allows a rigorous interpretation of nonparametric MAP estimation as Tikhonov--Phillips regularization and vice versa. 

Due to the correspondence of weak posterior modes to minimizers of the Onsager--Machlup functional described above, we moreover obtain equivalence of weak and strong posterior modes under the assumptions of \cref{var_char_modes}.

\begin{corollary}
	\label{weak_strong_modes}
	Let \cref{ass:post_dist,ass:phi} hold.
	Then $x \in X$ is a mode of $\mu^y$ if and only if it is a weak mode of $\mu^y$.
\end{corollary}

\subsection{Proof of the main result}

In \cite{Dashti:2013}, the coincidence of posterior modes with minimizers of the Onsager--Machlup functional is stated in a separable Banach space setting as Theorem 3.5 and Corollary 3.10. However, the proof given in \cite{Dashti:2013} is incomplete. On the one hand, parts of the proof only hold for separable Hilbert spaces, as pointed out in section 1.1 of \cite{klebanov2023}. On the other hand, even in the Hilbert space case the proof contains gaps that are closed in this work. Our proof follows the fundamental approach of \cite{Dashti:2013}, incorporating corrections where necessary. Most of these corrections have been introduced in \cite{Kretschmann:2019}, while some were just recently found necessary in \cite{klebanov2023}.
The outline of the proof of \cref{var_char_modes} can be described as follows.
\begin{enumerate}
	\item \label{amf} For every $\epsilon > 0$, choose $x_\epsilon \in X$ such that
	\begin{equation}
		\label{def:amf}
		\mu^y(B_\epsilon(x_\epsilon)) \ge (1 - \epsilon) \sup_{x \in X} \mu^y(B_\epsilon(x)).
	\end{equation}
	\item \label{conv_subseq_E} Show as follows that for every positive null sequence $\{\epsilon_n\}_{n\in\N}$, the sequence $\{x_{\epsilon_n}\}_{n\in\N}$ contains a subsequence that converges towards some $\bar{x} \in E$.
	\begin{enumerate}
		\item Show that $\{x_{\epsilon_n}\}_{n\in\N}$ is bounded and thus has a weakly convergent subsequence with limit $\bar{x} \in X$.
		\item Prove that $\bar{x} \in E$.
		\item Conclude that the subsequence converges in fact strongly toward $\bar{x}$.
	\end{enumerate}
	\item \label{cp_min_mode} Show that every cluster point of the sequence $\{x_{\epsilon_n}\}_{n\in\N}$ is both a mode of $\mu^y$ and a minimizer of the Onsager--Machlup functional of $\mu^y$.
	\item Use the existence of a point with these properties to prove that every mode of $\mu^y$ is a minimizer of the Onsager--Machlup functional of $\mu^y$ and vice versa.
\end{enumerate}
Points $x_\epsilon$ with property \eqref{def:amf} are called an \emph{asymptotic maximizing family}. Parts \ref{amf} to \ref{cp_min_mode} are stated separately as \cref{thmConvSub}.
The most notable corrections in comparison with \cite{Dashti:2013} are the following.
\begin{itemize}
	\item[(C1)] In step (ii) (c), it is used that the Gaussian prior $\mu_0$ satisfies
	\[ \liminf_{n \to \infty} \frac{\mu_0(B_{\epsilon_n}(x_n))}{\mu_0(B_{\epsilon_n}(0))} = 0 \]
	for any sequence $(x_n)_{n\in\N}$ in $X$ that converges weakly but not strongly towards some $\bar{x} \in E$ and any positive null sequence $(\epsilon_n)_{n\in\N}$. In \cite{Dashti:2013}, this statement was only proven for the special case $\bar{x} = 0$, see \cite[Lemma 3.9]{Dashti:2013}. The required statement was proven as Lemma 4.13 in \cite{Kretschmann:2019} and its proof was later simplified in Corollary 3.7 of \cite{klebanov2023}.
	\item[(C2)] In step (iii), it is used without proof that the Gaussian prior $\mu_0$ satisfies
	\[ \limsup_{n \to \infty} \frac{\mu_0(B_{\epsilon_n}(x_n))}{\mu_0(B_{\epsilon_n}(\bar{x}))} \le 1 \]
	for any sequence $(x_n)_{n\in\N}$ in $X$ that converges strongly towards some $\bar{x} \in E$ and any positive null sequence $(\epsilon_n)_{n\in\N}$. This statement was later proven by Masoumeh Dashti (personal communication, 3 July 2017), see Lemma 4.14 in \cite{Kretschmann:2019}.
	\item[(C3)] In step (iii), it is only proved that
	\[ \lim_{n\to\infty} \frac{\mu^y(B_{\epsilon_n^\prime}(\bar{x}))}{\sup_{x \in X} \mu^y(B_{\epsilon_n^\prime}(x))} = 1 \]
	for a subsequence $\{\epsilon_n^\prime\}_{n\in\N}$ of $\{\epsilon_n\}_{n\in\N}$. This is not immediately obvious because it is hidden by the notation $\epsilon \to 0$. In order for $\bar{x}$ to be a mode, the limit needs to be $1$ for any positive sequence $\{\epsilon_n\}_{n\in\N}$ with $\epsilon_n \to 0$.
	\item[(C4)] In the proof of Theorem 3.5 in \cite{Dashti:2013}, a family $\{\tilde{x}_r\}_{r > 0}$ of maximizers of $x \mapsto \mu^y(B_r(x))$ is used instead of the family $\{x_\epsilon\}_{\epsilon > 0}$ defined in \eqref{def:amf}. However, as discussed in \cite{klebanov2023}, these maximizers --- also called \emph{radius-$r$ modes} --- do not necessarily exist. Their existence in our setting has been proven in Corollary A.9 of \cite{lambley2022}. Nonetheless, we adopt the more general approach of Klebanov and Wacker used in \cite{klebanov2023} to resolve this issue and work with an asymptotic maximizing family, which always exists by definition of the supremum.
\end{itemize}
In Theorems 2.4 and 2.5 of \cite{klebanov2023}, the coincidence of MAP estimates with minimizers of the Onsager--Machlup functional is proven for posterior distributions based upon diagonal Gaussian priors on $\ell^p$, $1 \le p < \infty$, under the same (stricter) assumptions on the likelihood as in \cite{Dashti:2013}. The proof of these results follows the fundamental structure of that in \cite{Dashti:2013,Kretschmann:2019}. The main differences lie in the use of an asymptotic maximizing family instead of radius-$r$ modes and in simplifying the proofs of several lemmas regarding the small ball probabilities of Gaussian measures in the separable Hilbert space case $p = 2$.
The proof of our main result, \cref{var_char_modes}, builds on the following auxiliary result.
\begin{theorem}
	\label{thmConvSub}
	For every $\varepsilon > 0$, let $x_\epsilon \in X$ satisfy
	\begin{equation*}
		\mu^y(B_\varepsilon(x_\varepsilon)) \ge \left(1 - \epsilon\right) \sup_{x \in X} \mu^y(B_\varepsilon(x)).
	\end{equation*}
	If $\Phi$ satisfies \cref{ass:phi}, then the following holds true for every positive sequence $\{\varepsilon_n\}_{n\in\N}$ with $\varepsilon_n \to 0$:
	\begin{enumerate}
		\item \label{conv_subseq} The sequence $\{x_{\varepsilon_n}\}_{n\in\N}$ contains a subsequence that converges strongly in $X$, and every weak cluster point $\bar{x}$ of $\{x_{\epsilon_n}\}_{n\in\N}$ satisfies $\bar{x} \in E$.
		\item \label{cluster_point_min} Every cluster point $\bar{x} \in E$ of $\{x_{\varepsilon_n}\}_{n\in\N}$ minimizes the Onsager--Machlup functional of $\mu^y$, and every subsequence $\{x_{\epsilon_{n_m}}\}_{m\in\N}$ converging toward $\bar{x}$ satisfies
		\begin{equation*}
			\lim_{m\to\infty} \frac{\mu^y(B_{\varepsilon_{n_m}}(\bar{x}))}{\mu^y(B_{\varepsilon_{n_m}}(x_{\varepsilon_{n_m}}))} = 1.
		\end{equation*}
		\item \label{cluster_point_map} Every cluster point $\bar{x} \in E$ of $\{x_{\varepsilon_n}\}_{n\in\N}$ is a mode of $\mu^y$.
	\end{enumerate}
\end{theorem}

\begin{remark}
	\Cref{thmConvSub} guarantees in particular the existence of a mode.
\end{remark}

\begin{proof}[Proof of \cref{thmConvSub}]
First of all, we note that for each $\epsilon > 0$, an $x_\epsilon \in X$ with $\mu^y(B_\epsilon(x_\epsilon)) \ge (1 - \epsilon) \sup_{x \in X} \mu^y(B_\epsilon(x))$ exists by definition of the supremum, and the supremum is finite because $\mu^y$ is a probability measure.
Without loss of generality, we may assume that $\Phi(0) = 0$, because adding a constant to $\Phi$ can be absorbed by the normalization constant $Z$ without changing the measure $\mu^y$.

Ad \ref{conv_subseq}: Let $\{\epsilon_n\}_{n\in\N}$ be a positive null sequence and assume w.l.o.g. that $\epsilon_n \le 1$ for all $n \in \N$. We first show that $\{x_{\epsilon_n}\}_{n\in\N}$ is bounded in $X$.
By \cref{ass:phi} \ref{ass:phi_loc_lip}, we have
	\[ \Phi(u) = \Phi(u) - \Phi(0) \le L(1) \norm{u}_X \]
for all $u \in B_1(0)$. From this we obtain, using Anderson's inequality, that
\begin{align*} 
	\mu^y(B_\varepsilon(x_\epsilon))
	&\ge \left(1 - \epsilon\right) \sup_{u \in X} \int_{B_\varepsilon(u)} \mu^y(dv)
	= \left(1 - \epsilon\right) \sup_{u \in X} \int_{B_\varepsilon(u)} \frac{1}{Z} e^{-\Phi(v)} \mu_0(dv) \nonumber\\
	&\ge \frac{1 - \epsilon}{Z} \int_{B_\varepsilon(0)} e^{-\Phi(v)} \mu_0(dv)
	\ge \frac{1 - \epsilon}{Z} \int_{B_\varepsilon(0)} e^{-L(1)\norm{v}_X} \mu_0(dv) \nonumber\\
	&\ge \frac{1 - \epsilon}{Z} e^{-L(1)\varepsilon} \mu_0(B_\varepsilon(0))
\end{align*}
with $Z = \int_X \exp(-\Phi(v)) \mu_0(\mathrm{d}v)$ for all $\epsilon \in (0,1]$.
On the other hand, we have
	\[ \Phi(u) = \Phi(u) - \Phi(0) \ge \underline{L} \norm{u}_X \]
for all $u \in X$ by \cref{ass:phi} \ref{ass:phi_lower_bound}, which yields
\begin{equation*} 
	\mu^y(B_\varepsilon(u)) 
	= \int_{B_\varepsilon(u)} \frac{1}{Z} e^{-\Phi(v)} \mu_0(dv)
	\le \frac{1}{Z} \int_{B_\varepsilon(u)} e^{\underline{L}\norm{v}_X} \mu_0(dv)
	\le \frac{1}{Z} e^{\underline{L}\left(\norm{u}_X + \varepsilon\right)} \mu_0(B_\varepsilon(u))
\end{equation*}
holds for all $u \in X$.
Together, we obtain
\begin{equation}
	\label{eq:gauss_est1}
	\mu_0(B_\varepsilon(x_\epsilon)) 
	\ge Z e^{-\underline{L}\left(\left\|x_\epsilon\right\|_X + \varepsilon\right)} \mu^y(B_\varepsilon(x_\epsilon))
	\ge \left(1 - \epsilon\right) e^{-\underline{L}\left\|x_\epsilon\right\|_X - \left(\underline{L} + L(1)\right)\varepsilon} \mu_0(B_\varepsilon(0))
\end{equation}
for all $\varepsilon \in (0,1]$.
However, \cite[Lem.~3.6]{Dashti:2013} proves that there exists $a_1 > 0$ such that
\begin{equation}
	\label{eq:gauss_est2}	
	\frac{\mu_0(B_\varepsilon(x_\epsilon))}{\mu_0(B_\varepsilon(0))} \le e^{-\frac{a_1}{2}\left(\left\|x_\epsilon\right\|_X^2 - 2\varepsilon\right)}
\end{equation}
for all $\varepsilon > 0$.
Assuming that $\{x_{\epsilon_n}\}_{n\in\N}$ is unbounded, i.e., that there is a subsequence, again denoted by $\{x_{\epsilon_n}\}_{n\in\N}$, with $\norm{x_{\epsilon_n}}_X \to \infty$ as $n \to \infty$, leads to a contradiction, because in this case,
\begin{align*}
	&\frac{a_1}{2}\left(\norm{x_{\epsilon_n}}^2 - 2\varepsilon_n\right) - \underline{L}\norm{x_{\epsilon_n}} - \left(\underline{L} + L(1)\right)\varepsilon_n + \ln(1 - \epsilon_n) \\
	&= \left(\frac{a_1}{2}\norm{x_{\epsilon_n}} - \underline{L}\right)\norm{x_{\epsilon_n}} - \left(a_1 + \underline{L} + L(1)\right)\varepsilon_n + \ln(1 - \epsilon_n)\to \infty
\end{align*}
as $n \to \infty$, which implied
	\[ \frac{\mu_0(B_\varepsilon(x_\epsilon))}{\mu_0(B_\varepsilon(0))} 
	\ge \left(1 - \epsilon_n\right) e^{-\underline{L}\left\|x_{\epsilon_n}\right\|_X - \left(\underline{L} + L(1)\right)\varepsilon_n} 
	> e^{-\frac{a_1}{2}\left(\left\|x_{\epsilon_n}\right\|_X^2 - 2\varepsilon_n\right)} 
	\ge \frac{\mu_0(B_\varepsilon(x_\epsilon))}{\mu_0(B_\varepsilon(0))} \]
for sufficiently large $n$ by \eqref{eq:gauss_est1} and \eqref{eq:gauss_est2}.
So $\{x_{\epsilon_n}\}_{n\in\N}$ is bounded.
By the reflexivity of $X$, it therefore contains a subsequence $\{x_{\epsilon_n}\}_{n \in N_1}$ that converges weakly toward some $\bar{x} \in X$.

Now, we show that $\tilde{x} \in E$ for any weak cluster point $\tilde{x} \in X$ of $\{x_{\epsilon_n}\}_{n\in\N}$.
We can choose $R > 0$ such that $x_{\epsilon_n} \in B_R(0)$ for all $n \in \N$. Let $L$ denote the Lipschitz constant of $\Phi$ on $B_{R + 1}(0)$.
By the choice of $x_\epsilon$ and the boundedness of $\{x_{\epsilon_n}\}_{n \in \N}$, we have
\begin{equation}
\label{eq:est_ratio_mu0}
\begin{split}
	\frac23 &\le 1 - \epsilon \le \frac{\mu^y(B_\varepsilon(x_\epsilon))}{\mu^y(B_\varepsilon(0))}
	= \frac{\int_{B_\varepsilon(x_\epsilon)} e^{-\Phi(v)} \mu_0(\di v)}{\int_{B_\varepsilon(0)} e^{-\Phi(v)} \mu_0(\di v)}
	\le \frac{e^{L\left(\left\|x_\epsilon\right\|_X + \varepsilon\right)}}{e^{-L\varepsilon}} \frac{\mu_0(B_\varepsilon(x_\epsilon))}{\mu_0(B_\varepsilon(0))} \\
	&= e^{L\left(\left\|x_\epsilon\right\|_X + 2\varepsilon\right)} \frac{\mu_0(B_\varepsilon(x_\epsilon))}{\mu_0(B_\varepsilon(0))}
	\le e^{L\left(R + 2\right)} \frac{\mu_0(B_\varepsilon(x_\epsilon))}{\mu_0(B_\varepsilon(0))}
\end{split}
\end{equation}
for all $0 < \varepsilon \le \frac13$.
Let $\{x_{\epsilon_n}\}_{n \in N_2}$ denote the subsequence that converges weakly toward $\tilde{x}$.
If we assume that $\tilde{x} \notin E$, then \cite[Lem.~3.7]{Dashti:2013} tells us that there exists $n \in N_2$ such that $\mu_0(B_{\varepsilon_n}(x_{\epsilon_n}))/\mu_0(B_{\varepsilon_n}(0)) \le \frac13 e^{-L\left(R + 2\right)}$, and consequently
	\[ \frac{\mu^y(B_{\varepsilon_n}(x_{\epsilon_n}))}{\mu^y(B_{\varepsilon_n}(0))} \le \frac13, \]
which poses a contradiction. So $\tilde{x} \in E$ and, consequently, also $\bar{x} \in E$.

Next, we show that $\{x_{\epsilon_n}\}_{n \in N_1}$ converges strongly toward $\bar{x} \in E$. Suppose it does not.
Then Lemma 4.13 in \cite{Kretschmann:2019} (or Corollary 3.7 in \cite{klebanov2023}) applies and yields the existence of an $n \in N_1$ such that
	\[ \frac{\mu_0(B_{\varepsilon_n}(x_{\epsilon_n}))}{\mu_0(B_{\varepsilon_n}(0))} \le \frac13 e^{-L(R + 2)}, \]
which contradicts \eqref{eq:est_ratio_mu0}.
So the subsequence $\{x_{\epsilon_n}\}_{n \in N_1}$ does indeed converge strongly in $X$ toward $\bar{x} \in E$.

Ad \ref{cluster_point_min}: Let $\{x_{\epsilon_n}\}_{n\in\N}$ denote a subsequence that converges toward the cluster point $\bar{x} \in E$.
First, we show that
	\[ \lim_{n\to\infty} \frac{\mu^y(B_{\varepsilon_n}(\bar{x}))}{\mu^y(B_{\varepsilon_n}(x_{\epsilon_n}))} = 1. \]
By definition of $x_\epsilon$ and the Lipschitz continuity of $\Phi$ on bounded sets we have
\begin{equation*}
	1 - \epsilon \le \frac{\mu^y(B_\varepsilon(x_\epsilon))}{\mu^y(B_\varepsilon(\bar{x}))}
	= e^{\Phi(\bar{x}) - \Phi(x_\epsilon)} \frac{\int_{B_\varepsilon(x_\epsilon)} e^{\Phi(x_\epsilon)-\Phi(v)} \mu_0(\di v)}{\int_{B_\varepsilon(\bar{x})} e^{\Phi(\bar{x})-\Phi(v)} \mu_0(\di v)}
	\le e^{L \norm{x_\epsilon - \bar{x}}_X} e^{2L\varepsilon} \frac{\mu_0(B_{\varepsilon}(x_\epsilon))}{\mu_0(B_{\varepsilon}(\bar{x}))},
\end{equation*}
for all $\varepsilon > 0$ and consequently, by the convergence $x_{\epsilon_n} \to \bar{x}$ and Lemma 4.14 in \cite{Kretschmann:2019} (or Lemma A.2 in \cite{klebanov2023}),
\begin{equation*}
	1 \le \liminf_{n\to\infty} \frac{\mu^y(B_{\varepsilon_n}(x_{\epsilon_n}))}{\mu^y(B_{\varepsilon_n}(\bar{x}))}
	\le \limsup_{n\to\infty} \frac{\mu^y(B_{\varepsilon_n}(x_{\epsilon_n}))}{\mu^y(B_{\varepsilon_n}(\bar{x}))}
	\le \limsup_{n\to\infty} \frac{\mu_0(B_{\varepsilon_n}(x_{\epsilon_n}))}{\mu_0(B_{\varepsilon_n}(\bar{x}))} \le 1.
\end{equation*}

Next, we show that $\bar{x}$ minimizes the Onsager--Machlup functional $I$ of $\mu^y$. By \cref{min_I_ex}, a minimizer $x^* \in E$ of $I$ exists. If we suppose that $\bar{x}$ was not a minimizer of $I$, then $I(\bar{x}) - I(x^*) > 0$, and thus
\begin{align*}
	1 &= \lim_{n \to \infty} (1 - \epsilon_n) \le \lim_{n \to \infty} \frac{\mu^y(B_{\varepsilon_n}(x_{\epsilon_n}))}{\mu^y(B_{\varepsilon_n}(x^*))} 
	= \lim_{n \to \infty} \frac{\mu^y(B_{\varepsilon_n}(x_{\epsilon_n}))}{\mu^y(B_{\varepsilon_n}(\bar{x}))} \lim_{n \to \infty} \frac{\mu^y(B_{\varepsilon_n}(\bar{x}))}{\mu^y(B_{\varepsilon_n}(x^*))} \\
	&= 1\exp(I(x^*) - I(\bar{x})) < 1,
\end{align*}
by the choice of $x_\epsilon$ and \cref{form_OM_funct}, which poses a contradiction.

Ad \ref{cluster_point_map}: It remains to show that $\bar{x}$ is a mode of $\mu^y$, i.e., that for every positive sequence $\{\delta_n\}_{n\in\N}$ with $\delta_n \to 0$, we have
\begin{equation}
	\label{eq:map_prop_ubar}
	\lim_{n\to\infty} \frac{\mu^y(B_{\delta_n}(\bar{x}))}{\mu^y(B_{\delta_n}(x_{\delta_n}))} = 1.
\end{equation}
To this end, we choose an arbitrary subsequence of $\{\delta_n\}_{n\in\N}$, again denoted by $\{\delta_n\}_{n\in\N}$.
Then, by \ref{conv_subseq}, there exists a subsubsequence, again denoted by $\{\delta_n\}_{n\in\N}$, such that $x_{\delta_n} \to \tilde{x}$ for some $\tilde{x} \in E$. Moreover, by \ref{cluster_point_min}, the limit $\tilde{x}$ minimizes $I$ and satisfies
	\[ \lim_{n\to\infty} \frac{\mu^y(B_{\delta_n}(\tilde{x}))}{\mu^y(B_{\delta_n}(x_{\delta_n}))} = 1. \]
Since $\bar{x}$ minimizes $I$ as well by \ref{cluster_point_min}, this implies
\begin{align*}
	\lim_{n \to \infty} \frac{\mu^y(B_{\delta_n}(\bar{x}))}{\mu^y(B_{\delta_n}(x_{\delta_n}))}
	&= \lim_{n \to \infty} \frac{\mu^y(B_{\delta_n}(\bar{x}))}{\mu^y(B_{\delta_n}(\tilde{x}))}
	\lim_{n\to\infty} \frac{\mu^y(B_{\delta_n}(\tilde{x}))}{\mu^y(B_{\delta_n}(x_{\delta_n}))} \\
	&= \exp(I(\tilde{x}) - I(\bar{x}))1 = 1
\end{align*}
for the subsubsequence $\{\delta_n\}_{n\in\N}$ by \cref{form_OM_funct}.
Now \eqref{eq:map_prop_ubar} follows for the original sequence $\{\delta_n\}_{n\in\N}$ from a subsequence-subsequence argument.
\end{proof}

Now, we are able to prove the main result.
\begin{proof}[Proof of \cref{var_char_modes}]
By \cref{thmConvSub}, there exists $\bar{x} \in E$ which is both a mode of $\mu^y$ and a minimizer of the Onsager--Machlup functional $I$ of $\mu^y$.

Let $\hat{x}$ be a mode of $\mu^y$. Since $\bar{x}$ is also a mode, it follows that
\begin{align*}
	\lim_{\varepsilon \to 0} \frac{\mu^y(B_\varepsilon(\hat{x}))}{\mu^y(B_\varepsilon(\bar{x}))}
	&= \lim_{\varepsilon \to 0} \frac{\mu^y(B_\varepsilon(\hat{x}))}{\sup_{x \in X} \mu^y(B_\varepsilon(x))} \lim_{\varepsilon \to 0} \frac{\sup_{x \in X} \mu^y(B_\varepsilon(x))}{\mu^y(B_\varepsilon(\bar{x}))} = 1.
\end{align*}
Due to the Lipschitz continuity of $\Phi$ on bounded sets, we have
\begin{equation*}
	\frac{\mu^y(B_\varepsilon(\hat{x}))}{\mu^y(B_\varepsilon(\bar{x}))}
	= e^{\Phi(\bar{x}) - \Phi(\hat{x})} \frac{\int_{B_\varepsilon(\hat{x})} e^{\Phi(\hat{x})-\Phi(v)} \mu_0(\di v)}{\int_{B_\varepsilon(\bar{x})} e^{\Phi(\bar{x})-\Phi(v)} \mu_0(\di v)}
	\le e^{L \norm{\hat{x} - \bar{x}}_X} e^{2L\varepsilon} \frac{\int_{B_\varepsilon(\hat{x})} \mu_0(\di v)}{\int_{B_\varepsilon(\bar{x})} \mu_0(\di v)}.
\end{equation*}
This implies that $\hat{x} \in E$, as otherwise \cite[Lem.~3.7]{Dashti:2013} lead to
\begin{align*}
	1 &= \liminf_{\varepsilon \to 0} \frac{\mu^y(B_\varepsilon(\hat{x}))}{\mu^y(B_\varepsilon(\bar{x}))}
	\le e^{L \norm{\hat{x} - \bar{x}}_X} \liminf_{\varepsilon \to 0} \frac{\mu_0(B_\varepsilon(\hat{x}))}{\mu_0(B_\varepsilon(\bar{x}))} = 0,
\end{align*}
a contradiction. The definition of the OM functional yields
\begin{align*}
	1 = \lim_{\varepsilon \to 0} \frac{\mu^y(B_\varepsilon(\hat{x}))}{\mu^y(B_\varepsilon(\bar{x}))}
	&= \exp(I(\bar{x}) - I(\hat{x})),
\end{align*}
and consequently $I(\bar{x}) = I(\hat{x})$, i.e., $\hat{u}$ is also a minimizer of $I$.

Conversely, let $x^* \in E$ be a minimizer of $I$. Since $\bar{x}$ is also a minimizer of $I$ and a mode of $\mu^y$, \cref{form_OM_funct} tells us that
\begin{equation*}
	\lim_{\varepsilon \to 0} \frac{\mu^y(B_\varepsilon(x^*))}{\sup_{x \in X} \mu^y(B_\varepsilon(x))}
	= \lim_{\varepsilon \to 0} \frac{\mu^y(B_\varepsilon(x^*))}{\mu^y(B_\varepsilon(\bar{x}))} \lim_{\varepsilon \to 0} \frac{\mu^y(B_\varepsilon(\bar{x}))}{\sup_{x \in X} \mu^y(B_\varepsilon(x))}
	= \exp(I(\bar{x}) - I(x^*))1 = 1,
\end{equation*}
i.e., $x^*$ is also a mode of $\mu^y$.
\end{proof}


\section{Bayesian inverse problems}
\label{sect_ip}

Let us now consider the inverse problem of finding an unknown quantity $u$, given an indirect measurement
\begin{equation}
	\label{inv_prob}
	y = F(u) + \eta
\end{equation}
corrupted by additive noise $\eta$. Here, $u$ and $y$ each lie in a separable Hilbert space $X$ and $Y$, respectively. The relation between unknown and measured quantity is described by an operator $F$: $X \to Y$ and may be nonlinear.
We are particularly interested in the case when this problem is ill-posed.
We assume that the noise $\eta$ follows a distribution $\nu$.
Now, we take a Bayesian approach and assign a prior distribution $\mu_0$ to $u$, where we assume that $u$ and $\eta$ are independent. 

Let us assume that $\nu_z := \nu(\cdot - z)$ is absolutely continuous with respect to $\nu$ for all $z$ in a linear subspace $Z$ of $Y$, i.e., that $\{\nu_z\}_{z \in Z}$ is dominated by $\nu$. In this case, $\nu_z$ and $\nu$ are even equivalent since $\nu_{-z}$ is absolutely continuous w.r.t. to $\nu$ as well, and hence also $\nu = (\nu_z)_{-z}$ w.r.t. $\nu_z$.
If we furthermore assume that $\ran(F) \subseteq Z$, then $\nu_{F(u)}$ is a regular conditional distribution of $y$, given $u$, by \cite[Prop.~1.4]{Kretschmann:2019}.
The negative log-likelihood is then given by $\Psi(F(u),y)$, where
\[
	\Psi(z,y) := -\ln \frac{\di\nu_z}{\di\nu}(y)
\]
denotes for each $z \in Z$ the negative log-density of $\nu_z$ with respect to $\nu$. Here, the Radon--Nikodym derivative is nonzero $\nu$-almost everywhere due to the equivalence of $\nu_z$ and $\nu$.
If the likelihood $(u,y) \mapsto \exp(-\Psi(F(u),y))$ is a measurable function on $X \times Y$, and if the integral
\[
	\int_{X} \exp(-\Psi(F(x),y)) \mu_0(\di x)
\]
is positive for all $y \in Y$, then Bayes' formula 
\[
	\frac{\di\mu^y}{\di\mu_0}(u) 
	= \frac{\exp(-\Psi(F(u),y))}{\int_{X} \exp(-\Psi(F(x),y)) \mu_0(\di x)}
\]
defines a regular conditional distibution $\mu^y$ of $u$, given $y$, in terms of its density with respect to $\mu_0$ according to Bayes' theorem \cite[Thm.~1.3]{Kretschmann:2019}.
We see that $\mu^y$ satisfies \cref{ass:post_dist} \ref{ass:bayes_formula} with $\Phi(u,y) = \Psi(F(u),y)$.

Although the range condition $\ran(F) \subseteq Z$ is generally nontrivial in infinite-dimensional spaces, it is usually not restrictive in the context of inverse problems, since the forward operator is typically smoothing. Moreover, it is trivially satisfied in case of white Gaussian process noise, as will be discussed in the following subsection.
In \cref{lemSolEst}, we will see that the forward operator in case of the backward heat equation is infinitely smoothing. By Lemma 6.6 in \cite{kretschmann2022}, the forward operator in case of a deconvolution problem on $L^2(-1,1)$ with a specific kernel maps to the Sobolev space $H^4(-1,1)$.

\subsection{Gaussian noise}
\label{sec:gaussian_noise}

In case of nondegenerate centered Gaussian noise $\eta \sim \Normal_Q$ with covariance $Q$, $\nu_z = \Normal_{z,Q}$ is absolutely continuous with respect to $\nu = \Normal_Q$ for all $z \in Z = \ran(Q^{1/2})$ by the Cameron--Martin theorem \cite[Thm.~2.8]{DaPrato:2006}, and $\Psi$ is given by the Cameron--Martin formula
\begin{equation}
	\label{eq:Psi_gauss}
	\Psi(z,y) = \sum_{k=1}^\infty \left( \frac{\scalprod{z,e_k}_Y^2}{2\lambda_k} - \frac{\scalprod{z,e_k}_Y\scalprod{y,e_k}_Y}{\lambda_k} \right),
\end{equation}
where $(\lambda_k, e_k)_{k\in\N}$ denotes an eigensystem of $Q$. 

The following example shows that in case of an infinite-dimensional space $Y$, $\Phi(\cdot,y)$ is in general not bounded from below. That is, we can choose a forward mapping $F$, data $y \in Y$ and a sequence $(u^n)_{n\in\N}$ in $X$ such that $\Phi(u^n,y) \to -\infty$.
\begin{example}
\label{ex:gauss_unbounded}
Let $X$ and $Y$ be infinite-dimensional separable Hilbert spaces and let $(f_k)_{k\in\N}$ and $(e_k)_{k\in\N}$ be orthonormal bases of $X$ and $Y$, respectively. As a forward mapping, consider the bounded linear operator $A$: $X \to Y$ given by
\begin{equation*}
	Ax := \sum_{k=1}^\infty \frac{1}{k^2} \scalprod{x,f_k}_X e_k \quad \text{for all}~x \in X.
\end{equation*}
Solving $Ax = y$ can be interpreted as estimating the second derivative of a function $y \in C^2(0,1)$, see section 1.3.1 in \cite{Cavalier_2008}.
Consider, furthermore, Gaussian noise $\eta \sim \Normal_Q$ with covariance operator $Q$: $Y \to Y$,
\begin{equation*}
	Qz := \sum_{k=1}^\infty \frac{1}{k^2} \scalprod{z,e_k}_Y e_k = \left(AA^*\right)^\frac12 z \quad \text{for all}~z \in Y.
\end{equation*}
Then, $\Tr Q = \sum_{k=1}^\infty \frac{1}{k^2}$ is finite, and we have
\begin{align*}
	\ran(A) &= \left\{ z \in Y: \sum_{k=1}^\infty k^4 \scalprod{z,e_k}_Y^2 < \infty \right\} \\
	&\subseteq \left\{ z \in Y: \sum_{k=1}^\infty k^2 \scalprod{z,e_k}_Y^2 < \infty \right\} = \ran(Q^\frac12) = Z.
\end{align*}
The negative log-likelihood is given by
\begin{align*}
	\Phi(u,y) &= \Psi(Au,y) = \sum_{k=1}^\infty \left( \frac{\scalprod{Au,e_k}_Y^2}{2k^{-2}} - \frac{\scalprod{Au,e_k}_Y\scalprod{y,e_k}_Y}{k^{-2}} \right) \\
	&= \sum_{k=1}^\infty \left( \frac{\scalprod{Au - y,e_k}_Y^2}{2k^{-2}} - \frac{\scalprod{y,e_k}_Y^2}{2k^{-2}} \right)
\end{align*}
Now, we choose
\[ 
	y := \sum_{k=1}^\infty \frac{1}{k} e_k \qquad \text{and} \qquad
	u^n := \sum_{k=1}^n k f_k \quad \text{for all}~n \in \N.
\]
This way,
\[ \Phi(u^n,y) = \sum_{k=1}^n \left(0 - \frac{1}{2}\right) = -\frac{n}{2} \to -\infty \quad \text{as}~n \to \infty.\]
\end{example}

We examine under which conditions on the forward mapping $F$ the negative log-likelihood $\Phi(x,y) = \Psi(F(x),y)$ is Lipschitz continuous on bounded sets and satisfies the lower cone condition in case of additive Gaussian noise $\eta \sim \Normal_Q$. The majority of the proofs of this section are deferred to \cref{sect_proofs_ip}.

\begin{proposition}
	\label{cond_gauss}
	Assume that the forward mapping $F$ satisfies $\ran(F) \subseteq \ran(Q)$, that $Q^{-1} \circ F$ is Lipschitz continuous on bounded sets and that there exists $L_0 > 0$ such that
	\begin{equation}
		\label{Lip0_cond}
		\norm{Q^{-1}F(u) - Q^{-1}F(0)}_Y \le L_0 \norm{u}_X \quad \text{for all}~u \in X.
	\end{equation}
	Then $\Phi$: $X \times Y \to \R$, given by
	\begin{equation}
		\label{eq:Phi_gauss}
		\Phi(u,y) = \Psi(F(u),y) = \frac12\bignorm{Q^{-\frac12}F(u)}_Y^2 - \bigscalprod{Q^{-1}F(u),y}_Y,
	\end{equation}
	satisfies \cref{ass:phi}.
\end{proposition}

\begin{corollary}
	\label{cond_gauss_lin}
	If the forward mapping is a linear operator $A$ such that $\ran(A) \subseteq \ran(Q)$ and $Q^{-1}A$ is bounded, then $\Phi(u,y) = \Psi(Au,y)$ with $\Psi$ given by \eqref{eq:Psi_gauss} satisfies \cref{ass:phi}.
\end{corollary}

\begin{theorem}
	Let \cref{ass:post_dist} hold with $\Phi$ given by \eqref{eq:Phi_gauss}.
	Under the assumptions of \cref{cond_gauss} or \cref{cond_gauss_lin}, the modes of the posterior distribution $\mu^y$ are the minimizers of $I$: $\ran(\Sigma^{1/2}) \to \R$,
	\begin{equation}
		\label{eq:OM_funct_post_gauss}
		I(u) = \frac12\bignorm{F(u)}_Q^2 - \bigscalprod{Q^{-1}F(u),y}_Y + \frac12\norm{u}_\Sigma^2.
	\end{equation}
\end{theorem}
\begin{proof}
	By \cref{cond_gauss} or \cref{cond_gauss_lin}, $\Phi$ satisfies \cref{ass:phi}. This allows us to apply \cref{var_char_modes}, which tells us that the modes of $\mu^y$ coincide with the minimizers of its Onsager--Machlup functional. By \cref{form_OM_funct}, the Onsager--Machlup functional of $\mu^y$, in turn, is given by $u \mapsto \Phi(u,y) + \frac12\norm{u}_\Sigma^2$.
\end{proof}

\begin{example}
We return to \cref{ex:gauss_unbounded} and show that its set-up satisfies \cref{ass:phi}. The range of $A$
\begin{equation*}
	\ran(A) = \left\{ z \in Y: \sum_{k=1}^\infty k^4 \scalprod{z,e_k}_Y^2 < \infty \right\} = \ran(Q)
\end{equation*}
is equal to the range of $Q$, and we have
\[ Q^{-1}Ax = \sum_{k=1}^\infty \scalprod{x,f_k}_X e_k \quad \text{for all}~x \in X. \]
The operator $Q^{-1}A$ is bounded and therefore globally Lipschitz continuous. Moreover, condition \eqref{Lip0_cond} is satisfied with $L_0 = \norm{Q^{-1}A}_{X \to Y} = 1$.
Now, \cref{ass:phi} holds by \cref{cond_gauss}.
We see that here, $\Phi(\cdot,y)$ satisfies the lower cone condition but does not have a lower bound.
\end{example}

\subsubsection{White Gaussian noise}

Observations
\[ y = F(u) + \sigma w \]
perturbed by a Gaussian white noise process $\sigma w$ on $Y$, $\sigma > 0$, can be described rigorously by model \eqref{inv_prob}. A standard Gaussian white noise process $w$ almost surely takes values in the \emph{algebraic} dual space $Y^*$ of $Y$. It is characterized by the property that each marginal $\dualpair{w,g}_{Y^* \times Y}$, $g \in Y$, is a centered real Gaussian random variable and the marginals are corollated via $\mathbb{E}[\dualpair{w,g}_{Y^* \times Y}\dualpair{w,h}_{Y^* \times Y}] = \scalprod{g,h}_Y$ for all $g,h \in Y$.
For any self-adjoint, positive, trace class operator $Q$ on $Y$, $w$ moreover almost surely takes values in the \emph{topological} dual space $V' \supset Y'$ of $V := \ran(Q^{1/2}) \subset Y$, that is, realizations of $w$ are almost surely continuous linear functionals on $V$, see appendix 7.4 of \cite{Nickl2020Bernstein}, \cite{lassas2009}, or Remark 4.3 in \cite{kretschmann2022}. This allows interpreting $w$ rigorously as a $V'$-valued Gaussian random variable.

Let $\nu$ denote the distribution of $\sigma w$ on $V'$. Then, the shifted measure $\nu_{\scalprod{z,\cdot}_Y}$ is absolutely continuous w.r.t. $\nu$ for each $z \in Z := Y$ by the Cameron--Martin theorem \cite[Prop.~6.1.5]{gine_nickl_2015}, and
\begin{equation}
	\label{eq:Psi_gauss_white}
	\Psi(z,y) = -\ln \frac{\di\nu_{\scalprod{z,\cdot}}}{\di\nu}(y) = \frac{1}{2\sigma^2}\norm{z}_Y^2 - \frac{1}{\sigma^2}\dualpair{y,z}_{Y^* \times Y} \quad \nu\text{-a.e..}
\end{equation}

\begin{proposition}
	\label{cond_white_gauss}
	Let $Q$ be a self-adjoint, positive, trace class operator on $Y$. Let furthermore $V = \ran(Q^{1/2})$ be equipped with the norm $\norm{\cdot}_V = \norm{Q^{-1/2}\cdot}_Y$.
	If the forward mapping $F$ satisfies $\ran(F) \subseteq V$, is Lipschitz continuous on bounded sets as a mapping from $X$ to $V$, and there exists $L_0 > 0$ such that
	\[ \norm{F(u) - F(0)}_{V} \le L_0\norm{u}_X \quad \text{for all}~u \in X, \]
	then $\Phi$: $X \times V' \to \R$, given by
	\begin{equation*}
		\Phi(u,y) = \Psi(F(u),y) = \frac{1}{2\sigma^2}\norm{F(u)}_Y^2 - \frac{1}{\sigma^2}\dualpair{y,F(u)}_{V' \times V},
	\end{equation*}
	satisfies \cref{ass:phi}.
\end{proposition}

\begin{corollary}
	\label{cond_white_gauss_lin}
	If the forward mapping is an injective Hilbert--Schmidt operator $A$, then $\Phi(u,y) = \Psi(Au,y)$ with $\Psi$ given by \eqref{eq:Psi_gauss_white} satisfies \cref{ass:phi}.
\end{corollary}

\begin{remark}
	\label{rem:cond_white_gauss_lin}
	In case of a linear forward operator $A$, the assumptions of \cref{cond_white_gauss} simplify to the condition that there exists a self-adjoint, positive, trace class operator $Q$ on $Y$ such that $\ran(A) \subseteq \ran(Q^{1/2})$ and $Q^{-1/2}A$ is bounded.
	The assumption of \cref{cond_gauss_lin} that $Q^{-1}A$ is bounded is slightly stronger and not as straightforward to interpret since it depends on the given covariance structure of the noise.
	\Cref{lemSolEst} shows that the solution operator of the forward heat equation satisfies this condition in combination with Gaussian noise of the same covariance as the considered Laplacian noise.
\end{remark}

\subsection{Laplacian noise}
\label{sec:laplacian_noise}

We are particularly interested in the case of Laplacian noise, which we define as follows.
Let $\Laplace_{a,\lambda}$ denote the usual Laplace distribution on $\R$ with mean $a \in \R$ and variance $\lambda > 0$.
We define a Laplacian random variable $\xi$ with values in a separable Hilbert space $X$ for $a \in X$ and a self-adjoint, positive, trace class operator $Q \in L(X)$ with eigensystem $(\lambda_k,e_k)_{k\in\N}$ by
\begin{equation}
	\label{eq:def_xi}
	\xi := \sum_{k=1}^\infty \xi_k e_k,
\end{equation}
where $(\xi_k)_{k\in\N}$ are independent real-valued random variables with $\xi_k \sim \Laplace_{\scalprod{a,e_k}_X,\lambda_k}$ for all $k\in\N$. We denote its distribution by $\Laplace_{a,Q}$ or just $\Laplace_Q$ if $a = 0$. 

\begin{lemma}
	\label{xi_well_def}
	The series in \eqref{eq:def_xi} defines an $X$-valued random variable.
\end{lemma}

Note that in contrast with Gaussian random variables, a different choice of the eigenbasis $(e_k)_{k\in\N}$, if there exists a freedom of choice, defines a random variable with a different distribution.
Like Gaussian random variables, Laplacian random variables have the central property that for all elements of the chosen basis $(e_k)_{k\in\N}$, the marginals $\scalprod{\xi,e_k}_X = \xi_k$ are independent real valued Laplacian random variables. Arbitrary marginals $\scalprod{\xi,g}_X$, $g \in X$, are, however, in general not Laplacian.
This construction of Laplacian measures is similar to that of $1$-exponential measures in \cite{agapiou2021}, with the difference that it takes the covariance operator $Q$ as a starting point rather than the basis $(e_k)_{k\in\N}$.
It can also be considered a generalization of $B_1^s$-Besov measures from functions on $\R^d$ using a wavelet basis to arbitrary separable Hilbert spaces, cf.~\cite{Gerth_2014}.
For more details on the construction of Laplacian countable product measures, see also section 3.4 of \cite{Kretschmann:2019}.

The following result characterizes the admissible shifts of a Laplacian measure in terms of its covariance operator analogous to the Cameron--Martin theorem \cite[Thm.~2.8]{DaPrato:2006}. For its proof, see Theorem 3.10 in \cite{Kretschmann:2019} or Theorem 4.8 and Corollary 4.20 in \cite{Ayanbayev_2021b}.
\begin{theorem}
	\label{thmLaplaceAlt}
	Let $Q \in L(X)$ be a self-adjoint, positive, trace class operator with eigensystem $(\lambda_k, e_k)_{k\in\N}$ on a separable Hilbert space $X$ and $a \in X$. Let, moreover, $\Laplace_{a,Q}$ and $\Laplace_Q$ be defined using the eigenbasis $(e_k)_{k\in\N}$.
	\begin{enumerate}
		\item If $a \notin \ran(Q^\frac12)$, then $\Laplace_{a,Q}$ and $\Laplace_Q$ are singular.
		\item If $a \in \ran(Q^\frac12)$, then $\Laplace_{a,Q}$ and $\Laplace_Q$ are equivalent and the density $\frac{\di\Laplace_{a,Q}}{\di\Laplace_Q}$ is given by
		\begin{align*}
			\frac{\di\Laplace_{a,Q}}{\di\Laplace_Q}(x) 
			&= \exp\left( -\sqrt{2} \sum_{k=1}^\infty \left(\left|\scalprod{Q^{-\frac12}(x - a), e_k}_X\right| - \left|\scalprod{Q^{-\frac12}x, e_k}_X\right|\right) \right) \\
			&= \exp\left( -\sqrt{2}\sum_{k=1}^\infty \frac{ \left|\scalprod{x,e_k}_X - \scalprod{a,e_k}_X\right| - \left|\scalprod{x,e_k}_X\right| }{ \sqrt{\lambda_k} } \right)
		\end{align*}
	$\Laplace_Q$-almost everywhere.
	\end{enumerate}
\end{theorem}
\Cref{thmLaplaceAlt} shows that for Laplacian noise $\eta \sim \Laplace_Q$, $\nu_z = \Laplace_{z,Q}$ is absolutely continuous with respect to $\nu = \Laplace_Q$ for all $z \in Z = \ran(Q^{1/2})$ as in the Gaussian case, and
\begin{equation}
	\label{Psi_laplace}
	\Psi(z,y) = \sqrt{2} \sum_{k=1}^\infty \frac{\abs{\scalprod{y,e_k}_Y - \scalprod{z,e_k}_Y} - \abs{\scalprod{y,e_k}_Y}}{\lambda_k^{1/2}}.
\end{equation}

\begin{example}
	\label{ex:unbounded_phi}
	We return to \cref{ex:gauss_unbounded} and show that in case of Laplacian noise $\eta \sim \Laplace_Q$, $\Phi(\cdot,y)$ is not bounded from below either.
	Here, we have
	\begin{align*}
		\Phi(u^n,y) &= \Psi(Au^n,y) = \sqrt{2} \sum_{k=1}^\infty \frac{\abs{\scalprod{y - Au^n,e_k}_Y} - \abs{\scalprod{y,e_k}_Y}}{k^{-1}} \\
		&= \sqrt{2} \sum_{k=1}^n \left(0 - 1\right) = -\sqrt{2}n \to -\infty
	\end{align*}
	as $n$ tends to infinity for $y$ and $u^n$ chosen as in \cref{ex:gauss_unbounded}.
\end{example}

We examine under which conditions on the forward mapping $F$ the negative log-likelihood $\Phi(u,y) = \Psi(F(u),y)$ is Lipschitz continuous in $u$ in case of additive Laplacian noise $\eta \sim \Laplace_Q$.

\begin{proposition}
	\label{cond_nonlin_op}
	Let $(\lambda_k, e_k)_{k\in\N}$ denote an eigensystem of $Q$ and let the forward mapping $F$: $X \to Y$ satisfy one of the following two conditions.
	\begin{enumerate}
		\item There exists $C > 0$ such that
		\[ \sum_{k=1}^\infty \frac{\abs{\scalprod{F(x_1) - F(x_2), e_k}_Y}}{\lambda_k^{1/2}} \le C \norm{x_1 - x_2}_X \quad \text{for all}~x_1,x_2 \in X, \]
		\emph{or}
		\item $F$ is Fréchet differentiable and there exists $C > 0$ such that
		\begin{equation}
			\label{cond_F}
			\sum_{k=1}^\infty \frac{\norm{F'(x)^*e_k}_X}{\lambda_k^{1/2}} \le C \quad \text{for all}~x \in X.
		\end{equation}
	\end{enumerate}
	Then $\Phi$: $X \times Y \to \R$, given by
	\begin{equation}
		\label{eq:Phi_laplace}
		\Phi(u,y) = \Psi(F(u),y) = \sqrt{2} \sum_{k=1}^\infty \frac{\abs{\scalprod{y - F(u),e_k}_Y} - \abs{\scalprod{y,e_k}_Y}}{\lambda_k^{1/2}},
	\end{equation}
	is Lipschitz continuous in $u$ for all $y \in Y$.
	In particular, \cref{ass:phi} is satisfied.
\end{proposition}

\begin{proposition}
	\label{cond_lin_op}
	If the forward mapping is a compact linear operator $K$ such that $KK^*$ is diagonalizable with respect to the eigenvectors $(e_k)_{k\in\N}$ of $Q$ and $Q^{-1/2}K$ is a Hilbert--Schmidt operator, then $\Phi(u,y) = \Psi(Au,y)$ with $\Psi$ given by \eqref{Psi_laplace} is Lipschitz continuous in $u$ for all $y \in Y$ with Lipschitz constant $\sqrt{2}\norm{Q^{-1/2}K}_\text{HS}$ independent of $y$.
	In particular, \cref{ass:phi} is satisfied.
\end{proposition}

\begin{theorem}
	Let \cref{ass:post_dist} hold with $\Phi$ given by \eqref{eq:Phi_laplace}.
	Under the assumptions of \cref{cond_nonlin_op} or \cref{cond_lin_op}, the modes of the posterior distribution $\mu^y$ are the minimizers of $I$: $\ran(\Sigma^{1/2}) \to \R$,
	\[ I(u) = \sqrt{2} \sum_{k=1}^\infty \frac{\abs{\scalprod{y - F(u),e_k}_Y} - \abs{\scalprod{y,e_k}_Y}}{\lambda_k^{1/2}} + \frac12\norm{u}_\Sigma^2. \]
\end{theorem}
\begin{proof}
	By \cref{cond_nonlin_op} or \cref{cond_lin_op}, $\Phi$ satisfies \cref{ass:phi}. This allows us to apply \cref{var_char_modes}, which tells us that the modes of $\mu^y$ coincide with the minimizers of its Onsager--Machlup functional. By \cref{form_OM_funct}, the Onsager--Machlup functional of $\mu^y$, in turn, is given by $u \mapsto \Phi(u,y) + \frac12\norm{u}_\Sigma^2$.
\end{proof}

\begin{example}
	We return to \cref{ex:gauss_unbounded} and show that in case of Laplacian noise, $\Phi(\cdot,y)$ is Lipschitz continuous. To this end, we note that $Q^{-1/2}A$ is Hilbert--Schmidt because
	\[ \bignorm{Q^{-\frac12}A}_\text{HS}^2 = \bignorm{A^*Q^{-\frac12}}_\text{HS}^2 = \sum_{k=1}^\infty \norm{A^* Q^{-\frac12} e_k}_X^2 = \sum_{k=1}^\infty k^2 \norm{A^* e_k}_X^2 = \sum_{k=1}^\infty \frac{1}{k^2}. \]
	Thus, $\Phi(\cdot,y)$ is Lipschitz continuous by \cref{cond_lin_op}.
	We see that as in case of Gaussian noise, $\Phi(\cdot,y)$ satisfies the lower cone condition but is not bounded from below.
\end{example}

In \cref{prop:phi_cont}, we will see that the solution operator of the forward heat equation also satisfies the assumptions of \cref{cond_lin_op} in combination with the considered Laplacian noise.


\section{Severely ill-posed linear problems}
\label{sect_lin_prob}

Eventually, we study the rate of convergence of the MAP estimator based upon a Gaussian prior distribution as the noise level tends to zero for the example of a severely ill-posed linear problem with measurements corrupted by Laplacian noise.
This problem can be considered as a generalization of the inverse heat equation. It has been studied with measurements corrupted by Gaussian noise in subsections 1.2 and 3.3 of \cite{das_stu_2017}. For more details on this example, see also chapter 5 of \cite{Kretschmann:2019}.
The proofs of some of the results of this section are deferred to \cref{sect_proofs_lin_prob}.

\subsection{Set-up}

Throughout this section we consider the following set-up.
Let $\{\phi_k\}_{k\in\N}$ be an orthonormal basis of a separable Hilbert space $X$. Moreover, let $\{\alpha_k\}_{k\in\N}$ be a positive, non-decreasing sequence, $C_-, C_+ > 0$, and $d \in \N$ such that
\[
	C_- k^\frac{2}{d} \le \alpha_k \le C_+ k^\frac{2}{d} \quad \text{for all }k \in \N.
\]
Now, we define an unbounded linear operator $A$ on
\[
	\dom(A) := \left\{ x \in X: \sum_{k=1}^\infty \alpha_k^2 \abs{\scalprod{x,\phi_k}_X}^2 < \infty \right\}
\]
by
\begin{equation}
	\label{eq:def_A}
	Ax := \sum_{k=1}^\infty \alpha_k \scalprod{x,\phi_k}_X \phi_k.
\end{equation}
Next, we define the forward operator $K$: $X \to X$ via functional calculus as $K := e^{-A}$, i.e., by
\begin{equation}
	\label{eq:def_K}
	Kx := \sum_{k=1}^\infty e^{-\alpha_k} \scalprod{x,\phi_k}_X \phi_k.
\end{equation}
As a limit of linear operators with finite-dimensional range, $K$ is a compact linear operator.
We then consider the inverse problem of finding $u \in X$, given a measurement $y \in Y := X$ related to $u$ by
\begin{equation}
	\label{eq:lin_model}
	y = Ku + \eta,
\end{equation}
where $\eta \in X$ is noise.

\begin{example}
	For appropriately chosen subsets $\Omega$ of $\R^d$ (e.g. bounded, open, and with $C^\infty$ boundary), the weak Laplace operator $A := -\Delta$ in $X := L^2(\Omega)$ has the aforementioned properties, see Example 5.5 in \cite{Kretschmann:2019}. The eigenfunctions of $A$ under Dirichlet boundary conditions form an orthonormal basis of $L^2(\Omega)$ by Theorem 6.5.1 in \cite{Evans:1998}, and the corresponding eigenvalues grow in the order of $k^{2/d}$ according to Weyl's asymptotic formula (see Theorem 8.16 and Remark 8.17 in \cite{Roe:1998}). The domain of $A$ is then given by $\dom(A) = H^2(\Omega) \cap H_0^1(\Omega)$.
	
	For $u \in L^2(\Omega)$, $v(t) := e^{-tA}u$ constitutes the unique solution to the heat equation
	\[ \frac{\di v(t)}{\di t} = \Delta v(t) \quad \text{for}~t > 0 \]
	with initial condition $v(0) = u$ by Proposition 5.20 in \cite{Kretschmann:2019}. For $A = -\Delta$, the aforementioned inverse problem thus corresponds to finding the heat distribution $u$ on $\Omega$ at time $0$ from a noisy measurement $y$ of the heat distribution $Ku = e^{-A}u$ at time $1$.
\end{example}

For $s \in \R$, we define powers $A^s$ of $A$ on
\[
	\dom(A^{s}) := \left\{ x \in X: \sum_{k=1}^\infty \alpha_k^{2s} \abs{\scalprod{x,\phi_k}_X}^2 < \infty \right\}
\]
via functional calculus as
\[
	A^s x := \sum_{k=1}^\infty \alpha_k^s \scalprod{x,\phi_k}_X \phi_k.
\]
The operator $A$ induces a decreasing scale of Hilbert spaces $\{\X{s}\}_{s \ge 0}$, where
\begin{equation}
	\label{eq:def_Xs}
	\X{s} := \dom(A^\frac{s}{2}) = \ran(A^{-\frac{s}{2}}) = \left\{ x \in X: \sum_{k=1}^\infty \alpha_k^s \abs{\scalprod{x,\phi_k}_X}^2 < \infty \right\}
\end{equation}
is endowed with the norm $\norm{x}_\X{s} := \norm{A^{s/2}x}_X$.
The operator $A^{-s}$ is trace class for any $s > d/2$ since
\begin{equation}
	\label{eq:pow_A_trace_class}
	\Tr A^{-s} = \sum_{k=1}^\infty \scalprod{A^{-s}\phi_k,\phi_k}_X = \sum_{k=1}^\infty \alpha_k^{-s} \le C_-^{-s} \sum_{k=1}^\infty k^{-\frac{2s}{d}}.
\end{equation}

We model the inverse problem in a Bayesian framework. We assume that the noise follows a Laplacian distribution $\nu := \Laplace_{b^2A^{-\beta}}$ with $b > 0$ and $\beta > d/2$, and we assign a centered Gaussian prior distribution $\mu_0 := \Normal_{r^2A^{-\tau}}$ with $r > 0$ and $\tau > d/2$. 
Here, the measure $\Laplace_{b^2A^{-\beta}}$ is defined using the basis $\{\phi_k\}_{k\in\N}$.
As before, we assume that $u$ and $\eta$ are independent.

\subsection{Posterior distribution}

If the shifted measure $\Laplace_{e^{-A}u, b^2 A^{-\beta}}$ is absolutely continuous with respect to the noise distribution $\Laplace_{b^2 A^{-\beta}}$ for every $u \in X$, then it follows from \cite[Prop.~1.4]{Kretschmann:2019} that
\[ (u,V) \mapsto \Laplace_{e^{-A}u, b^2 A^{-\beta}}(V) \]
is a regular conditional distribution of $y$, given $u$.
\Cref{thmLaplaceAlt} tells us that $\Laplace_{e^{-A}u, b^2 A^{-\beta}}$ and $\Laplace_{b^2 A^{-\beta}}$ are equivalent if and only if $e^{-A}u \in \ran(A^{-\beta/2})$.
This is indeed the case for all $u \in X$ by the following lemma.
\begin{lemma}
	\label{lemSolEst}
	Let $s > 0$. Then $\ran(\exp(-tA)) \subseteq \X{s}$ for all $t > 0$ and there is a $C = C(s) > 0$, such that
	\[ \norm{\exp(-tA)u}_\X{s} \le C t^{-\frac{s}{2}}\norm{u}_X \quad \text{for all }t > 0, \]
	where $A$ is defined as in \eqref{eq:def_A} and $\X{s}$ as in \eqref{eq:def_Xs}.
\end{lemma}
From \cref{thmLaplaceAlt} we now obtain for each $u \in X$ the density
\begin{equation*}
	\frac{\di \Laplace_{e^{-A}u, b^2 A^{-\beta}}}{\di \Laplace_{b^2 A^{-\beta}}}(y)
	= \exp(-\Phi(u,y))
\end{equation*}
of the conditional distribution of $y$, given $u$, with respect to that of the noise $\eta$, where
\begin{equation} 
	\label{eq:def_phi}
	\Phi(u,y) := \sqrt{2} \sum_{k=1}^\infty \frac{ \left|\scalprod{y,\varphi_k}_X - e^{-\alpha_k}\scalprod{u,\varphi_k}_X\right| - \left|\scalprod{y,\varphi_k}_X\right|}{b\alpha_k^{-\beta/2}}
	\quad \text{for all}~u,y \in X.
\end{equation}

In order to derive the conditional distribution of $u$, given $y$, using Bayes' theorem, we record the following properties of $\Phi$.

\begin{proposition}
	\label{prop:phi_cont}
	The function $\Phi$: $X \times X \to \R$, defined by \eqref{eq:def_phi}, is continuous, and for every $y \in X$, $u \mapsto \Phi(u,y)$ is Lipschitz continuous with a Lipschitz constant independent of $y$.
\end{proposition}

\begin{proposition}
	\label{prop:Z_pos_fin}
	Let $r > 0$ and $\tau > d/2$.
	Then $u \mapsto \exp(-\Phi(u,y))$ is $\Normal_{r^2 A^{-\tau}}$-integrable for all $y \in X$ and there exists  $C_Z > 0$ such that
	\begin{equation*}
		\int_X \exp(-\Phi(u,y)) \Normal_{r^2 A^{-\tau}}(\di u) \ge C_Z \quad \text{for all }y \in X,
	\end{equation*}
	where $\Phi$ is defined as in \eqref{eq:def_phi}.
\end{proposition}

With this knowledge, we are able to apply Bayes' formula.
\begin{theorem}
	\label{post_dist_lin_prob}
	Assume a Gaussian prior distribution $\Normal_{r^2 A^{-\tau}}$ for $u$ with $A$ defined as in \eqref{eq:def_A}, $r > 0$, and $\tau > d/2$. Assume further that $(u, V) \mapsto \Laplace_{e^{-A}u, b^2 A^{-\beta}}(V)$ with $b > 0$ and $\beta > d/2$ is a regular conditional distribution of the data $y$, given $u$.
	Then a regular conditional distribution $(y,B) \mapsto \mu^y(B)$ of $u$, given $y$, exists, the posterior distribution $\mu^y$ is absolutely continuous with respect to the prior distribution $\Normal_{r^2 A^{-\tau}}$ for every $y \in X$ and has the density
	\begin{equation} 
		\label{eq:post_dens}
		\frac{\di\mu^y}{\di\Normal_{r^2 A^{-\tau}}}(u)
		= \frac{1}{Z(y)} \exp(-\Phi(u,y)),
	\end{equation}
	where $\Phi$ is given by \eqref{eq:def_phi} and
	\begin{equation*}
		Z(y) = \int_X \exp(-\Phi(u,y)) \Normal_{r^2 A^{-\tau}}(\di u).
	\end{equation*}
\end{theorem}
\begin{proof}[Proof of \cref{post_dist_lin_prob}]
	By \cref{thmLaplaceAlt,lemSolEst}, the measure $P_{u} := \Laplace_{e^{-A}u, b^2 A^{-\beta}}$ is absolutely continuous with respect to $\nu := \Laplace_{b^2 A^{-\beta}}$ for all $u \in X$ with the density $y \mapsto p_{u}(y) := \exp(-\Phi(u,y))$. 
	The function $(u,y) \mapsto p_u(y)$ is measurable by \cref{prop:phi_cont} and $Z(y) > 0$ for all $y \in X$ by \cref{prop:Z_pos_fin}.
	Therefore, we may apply Bayes' theorem \cite[Thm.~1.3]{Kretschmann:2019}, which yields the proposition.
\end{proof}

\Cref{post_dist_lin_prob} shows that $\mu^y$ satisfies \cref{ass:post_dist}, and \cref{prop:phi_cont} shows that $\Phi$ satisfies \cref{ass:phi}.
The posterior distribution is stable with respect to changes in the data, that is, small changes in $y$ result in small changes in the $\mu^y$ measured in Hellinger distance, see Theorem 5.30 in \cite{Kretschmann:2019}. Introducing a Gaussian prior has a stabilising effect on the considered inverse problem.

\subsection{Maximum a posteriori estimator}

Next, we use the variational characterization from \cref{var_char_modes} to find the modes of the posterior distribution $\mu^y$ as minimizers of its Onsager--Machlup functional.
By the Cameron--Martin theorem \cite[Thm.~2.8]{DaPrato:2006}, the Cameron--Martin space of the prior distribution $\Normal_{r^2 A^{-\tau}}$ is given by $\X{\tau} = \ran(A^{-\tau/2})$, and its Cameron--Martin norm by $\frac{1}{r}\norm{h}_\X{\tau} = \frac{1}{r}\norm{A^{\tau/2}h}_X$.

\begin{theorem}
	\label{explicit_minimizer}
	For every $y \in \X{s}$, the functional $I^y: \X{\tau} \to \R$, defined by
	\begin{equation}
		\label{eq:omf_heat}
		\begin{aligned}
			I_y(u) &= \Phi(u,y) + \frac{1}{2r^2}\norm{u}_\X{\tau}^2 \\
			&= \frac{\sqrt{2}}{b} \sum_{k=1}^\infty \frac{\abs{\scalprod{y,\phi_k}_X - e^{-\alpha_k}\scalprod{u,\phi_k}_X} - \abs{\scalprod{y,\phi_k}_X}}{\alpha_k^{-\beta/2}} + \frac{1}{2r^2} \sum_{k=1}^\infty \frac{\abs{\scalprod{u,\phi_k}_X}^2}{\alpha_k^{-\tau}},
		\end{aligned}
	\end{equation}
	is an Onsager--Machlup functional of $\mu^y$ as given in \eqref{eq:post_dens}. It has a unique minimizer $\bar{u} = \bar{u}(y) \in \X{\tau}$ that is the only mode of $\mu^y$ and is given by
	\begin{equation}
		\label{eq:explicit_minimizer}
		\bar{u} = \sum_{k=1}^\infty e^{\alpha_k} P_{\big[-\frac{r^2}{b}R_k,\frac{r^2}{b}R_k\big]} \Big(\scalprod{y,\varphi_k}_X\Big) \varphi_k,
	\end{equation}
	where
	\[ R_k := \sqrt{2}\alpha_k^{\frac{\beta}{2} - \tau}e^{-2\alpha_k} \quad \text{for all}~k \in \N, \]
	and
	\[ P_I(x) := \argmin_{z\in I} \abs{z - x} \]
	denotes the projection of $x \in \R$ onto a closed interval $I \subseteq \R$.
\end{theorem}
\begin{proof}
	Let $y \in X$. $\Phi$ is Lipschitz continuous in $u$ by \cref{prop:phi_cont}, so $u \mapsto \Phi(u,y) + \frac{1}{2r^2}\norm{u}_\X{\tau}^2$ is an Onsager--Machlup functional of $\mu^y$ by \cref{form_OM_funct}.
	Now, \cref{var_char_modes} tells us that $I^y$ has a minimizer $\bar u \in \X{\tau}$ and that the minimizers of $I^y$ are precisely the modes of $\mu^y$, again due to the Lipschitz continuity of $\Phi$ in $u$.
	A straightforward computation shows that the minimizer of $I^y$ is unique and given by \eqref{eq:explicit_minimizer}, see Lemma 5.34 in \cite{Kretschmann:2019}.
\end{proof}

\Cref{explicit_minimizer} tells us that for the considered inverse problem, a regularized solution found by minimizing $I^y$ as defined in \eqref{eq:omf_heat} does indeed describe a point of maximal posterior probability both in the sense of a small ball mode and a minimizer of the Onsager--Machlup functional.
In particular, \cref{explicit_minimizer} guarantees that for every $y \in X$, the posterior distribution $\mu^y$ has a unique mode. This allows us to define the \emph{maximum a posteriori estimator} $\umap$: $X \to X$ by assigning to each $y \in X$ the mode $\umap(y)$ of $\mu^y$, and to express it as
\begin{equation}
	\label{eq:def_map_estimator}
	\umap(y) = \argmin_{u \in \X{\tau}} I^y(u) = \argmin_{u \in \X{\tau}} \left\{ \Phi(u,y) + \frac{1}{2r^2}\norm{u}_\X{\tau}^2 \right\}.
\end{equation}
Minimizing $I^y$ corresponds to Tikhonov--Phillips regularization with an $\ell^1$-like discrepancy term and a quadratic penalty term, where the inverse prior variance plays the role of the regularization parameter.
We can interpret expression \eqref{eq:explicit_minimizer} for $\umap(y)$ in the following way: The MAP estimator acts upon the data by projecting it onto a hyperrectangle 
	\[ \left\{ y \in X: \abs{\scalprod{y, \varphi_k}_X} \le \frac{r^2}{b}R_k \text{ for all }k \in \N \right\} \]
and then applying the inverse $e^A$ of the forward operator.
It can be shown that the MAP estimator $\umap$ is continuous, see Theorem 5.37 in \cite{Kretschmann:2019}. That is, the MAP estimate is stable with respect to changes in the data.

\subsection{Consistency of the MAP estimator}

Finally, we show consistency of the MAP estimator as the noise level tends to zero under the frequentist assumption that the data is generated by a deterministic true solution $u^\dagger \in X$, that is, when
\[ y = e^{-A}u^\dagger + \eta \sim \Laplace_{e^{-A}u^\dagger, b^2 A^{-\beta}}. \]
Here, the value of our main result \cref{var_char_modes} lies in allowing us to translate the consistency of the minimizer of $I^y$ into consistency of the posterior mode.
We show that under a source condition on $u^\dagger$ and with an a priori choice of the prior variance $r^2$ in the order of the noise level $b$, the mean squared error (MSE) of the MAP estimator converges in the order of the noise level.

\begin{theorem} 
	\label{thmConvRateVar}
	Let $u^\dagger \in X$ and $r > 0$ and assume that $y \sim \Laplace_{e^{-A}u^\dagger, b^2 A^{-\beta}}$ with $b > 0$, $\beta > d/2$ and $A$ as defined in \eqref{eq:def_A}. 
	If there exist $w \in X$ and $\rho > 0$ such that
	\begin{equation}
		\label{eq:source_cond}
		u^\dagger = A^{\frac{\beta}{2} - \tau}e^{-A}w 
		\quad \text{and} \quad \sup_{k\in\N} \abs{\scalprod{w,\varphi_k}_X} \le \rho,
	\end{equation}
	and if there exists $C > 0$ such that
		\[ \frac{\rho}{\sqrt{2}} b \le r^2 \le C b, \]
	then 
		\[ \E\left[ \norm{\umap(y) - u^\dagger}_{X}^2 \right] \le 2C \left(\Tr A^{-\tau}\right) b \]
	for $\umap(y)$ given by \eqref{eq:def_map_estimator}.
\end{theorem}

\begin{proof}
Since the components of $\umap$ are independent by \cref{explicit_minimizer}, we have
\begin{equation*}
	\E\left[ \norm{\umap - u^\dagger}_{X}^2 \right]
	= \E\left[ \sum_{k=1}^\infty \left|\scalprod{\umap - u^\dagger, \varphi_k}_{X}\right|^2 \right]
	= \sum_{k=1}^\infty \E\left[ \left|\scalprod{\umap - u^\dagger, \varphi_k}_{X}\right|^2 \right].
\end{equation*}
By Lemma 5.49 in \cite{Kretschmann:2019}, the componentwise mean squared error is given by
\begin{equation}
	\label{eq:componentwise_MSE}
	\begin{aligned}
		&\E\left[\left|\scalprod{\umap(y) - u^\dagger, \varphi_k}_X\right|^2\right] \\
		&= \frac{b^2}{c_k^2} f\left(\frac{c_k}{b} \left| \frac{r^2}{b}c_k\alpha_k^{-\tau} + \left|\scalprod{u^\dagger, \varphi_k}_{X}\right| \right| \right)
		+ \frac{b^2}{c_k^2} f\left(\frac{c_k}{b} \left| \frac{r^2}{b}c_k\alpha_k^{-\tau} - \left|\scalprod{u^\dagger, \varphi_k}_{X}\right|\right| \right) \\
		&\qquad+ \mathbf{1}_{\left[-\frac{r^2}{b}c_k\alpha_k^{-\tau},\frac{r^2}{b}c_k\alpha_k^{-\tau}\right]} \left(\scalprod{u^\dagger,\varphi_k}_X\right)
		\frac{b^2}{c_k^2} g\left( \frac{c_k}{b} \left| \frac{r^2}{b}c_k\alpha_k^{-\tau} - \left|\scalprod{u^\dagger, \varphi_k}_{X}\right|\right| \right)
	\end{aligned}
\end{equation}
for all $k \in \N$, where $c_k := \sqrt{2}\alpha_k^{\beta/2}e^{-\alpha_k}$, $f(t) := 1 - e^{-t} - t e^{-t}$ and $g(t) := t^2 - 2f(t)$ for all $t \ge 0$. 
The conditions on $u^\dagger$ and $r$ ensure
\begin{align*}
	\left|\scalprod{u^\dagger, \varphi_k}_{X}\right|
	&= \left|\left(A^{\frac{\beta}{2} - \tau}e^{-A}w, \varphi_k\right)_{X}\right|
	= \left|\left(w, e^{-A}A^{\frac{\beta}{2} - \tau}\varphi_k\right)_{X}\right| \\
	&= e^{-\alpha_k}\alpha_k^{\frac{\beta}{2} - \tau} \left|\left(w,\varphi_k\right)_{X}\right|
	\le \alpha_k^{\frac{\beta}{2} - \tau}e^{-\alpha_k} \rho
	\le \frac{r^2}{b}\sqrt{2}\alpha_k^{\frac{\beta}{2} - \tau}e^{-\alpha_k}
\end{align*}
for all $k \in \N$ and $n \ge N$.
Thus, the last term on the right hand side of \eqref{eq:componentwise_MSE} is equal to zero.
We use the estimate 
	\[ f(t) \le 1 - e^{-t} \le t, \]
that holds for all $t \ge 0$, to obtain
	\[ \E\left[\left|\scalprod{\umap - u^\dagger, \varphi_k}_{X}\right|^2\right] \le 2 r^2 \alpha_k^{-\tau}. \]
Consequently, we have
	\[ \E\left[ \norm{\umap - u^\dagger}_{X}^2 \right] \le 2 r^2 \sum_{k=1}^\infty \alpha_k^{-\tau} = 2 \left(\Tr A^{-\tau}\right) r^2 \le 2C \left(\Tr A^{-\tau}\right) b \]
by the choice of $r$.
\end{proof}

\Cref{thmConvRateVar} shows in particular that with the stated choice of $r$, the MAP estimator is consistent, since its convergence toward the true solution in mean square implies convergence in probability by Markov's inequality.
If we sharpen the condition that $\sup_{k\in\N} \abs{\scalprod{w,\phi_k}_X} \le \rho$ to $\norm{w}_X \le \rho$, set $d = 2$ and are able to choose $\tau = \beta/2$, then \eqref{eq:source_cond} corresponds exactly to an analytic source condition. For $\tau < \beta/2$, it is weaker than an analytic source condition, whereas for $\tau > \beta/2$, it is stronger.
In numerical simulations, the MAP estimator has been observed to not converge to the truth $u^\dagger$ in the small noise limit if only a Sobolev-type source condition is fulfilled, see subsection 6.5.2 of \cite{Kretschmann:2019}, which suggests that a source condition of type \eqref{eq:source_cond} is necessary for convergence.

We compare the rate of convergence stated in \cref{thmConvRateVar} with the convergence rate of the minimax risk when the Laplacian noise is replaced by Gaussian noise $\eta \sim \Normal_{b^2 A^{-\beta}}$. 
For an introduction to the minimax paradigm and an overview over minimax rates of linear inverse problems with different degrees of ill-posedness and under different source conditions, see \cite{Cavalier_2008}. 
The minimax rate of the MSE for severely ill-posed problems under an analytic source condition is linear in the noise level $b$, see Table 1 in \cite{Cavalier_2008}, and it is shown that the MAP estimator achieves this rate in Proposition 10 of \cite{aga_mat_2018}. In our setting, this rate does, however, not apply because the noise is not white and the source condition is in general not exactly analytic.

We restrict ourselves to the case that the inverse problem is exponentially ill-posed, that is, we set $d = 2$, and assume that the eigenvalues of $A$ associated with the eigenvectors $\varphi_k$ are exactly $\alpha_k = pk$ for some $p > 0$.
We moreover replace the source condition \eqref{eq:source_cond} by the slightly stronger condition that the true solution lies in the class
\[ \Theta := \left\{ A^{\frac{\beta}{2} - \tau}e^{-A}w: w \in X, \norm{w}_X \le \rho \right\}. \]
This setting was considered in \cite{Tsybakov_2000} and corresponds to a choice of $\epsilon = b$,
\begin{equation*}
	b_k = k^\frac{\beta}{2} e^{-pk} \quad\text{and}\quad	
	a_k = k^{\tau - \frac{\beta}{2}} e^{pk} \quad \text{for all}~k \in \N
\end{equation*}
for the sequences $(b_k)_{k\in\N}$ and $(a_k)_{k\in\N}$ describing the model and the source condition.
Now, Theorem 1 in \cite{Tsybakov_2000} yields convergence of the MSE minimax risk
\[ \inf_{\hat{u}} \sup_{u^\dagger \in \Theta} \E\left[\norm{\hat{u} - u^\dagger}_X^2\right] \]
in the order of $(\log \frac{1}{b})^{-\tau} b$, where the infimum is taken over all estimators $\hat{u} = \hat{u}(y)$.
We observe that in the considered case, the convergence rate estimate for Laplacian noise from \cref{thmConvRateVar} is worse than the minimax rate for Gaussian noise by a logarithmic factor. The question remains whether the minimax rate is actually worse in this case or if the estimate is simply not sharp.


\section*{Conclusion}

We saw that under mild assumptions on the likelihood and in a separable Hilbert space setting, modes of a Bayesian posterior distribution based upon a Gaussian prior do indeed coincide with minimizers of its Onsager--Machlup functional and thus also with its weak modes. Under appropriate conditions on the forward mapping, our assumptions on the likelihood are in particular satisfied for Bayesian inverse problems with infinite-dimensional additive Gaussian or Laplacian noise. The proof of our main result fills gaps present in previous proofs and constitutes the first rigorous proof in the separable Hilbert space setting that covers the fundamental limit case of inverse problems with Gaussian process noise. 
Our work shows that in the considered cases, nonparametric MAP estimation and Tikhonov--Phillips regularization are equivalent and links the choice of the discrepancy term to the log-likelihood.
It moreover shows that weak and strong MAP estimates coincide in these cases.
In the considered example of a severely ill-posed linear problem with Laplacian noise, the variational characterization of modes allowed us to express the MAP estimator explicitly and bound the rate of convergence of its mean squared error.

Our results pave the way for the study of nonparametric MAP estimates for inverse problems with other non-Gaussian noise models.

\section*{Acknowledgments}

I want to thank Frank Werner and Christian Clason for their valuable feedback and fruitful discussions.
I furthermore want to thank the referees and editors for their comments and suggestions, which have helped improve this work significantly.


\appendix

\section{Proofs of section \ref{sect_var_char}}
\label{sect_proofs_var_char}

We first show that \cref{ass:phi} implies conditions on the likelihood used in \cite{Dashti:2013}.
\begin{lemma}
	\label{cond_Dashti_hold}
	If $\Phi$ satisfies \cref{ass:phi}, then it has the following three properties.
	\begin{enumerate}
		\item For every $\epsilon > 0$, there exists $M = M(\varepsilon) \in \R$, such that for all $x \in X$,
			\[ \Phi(x) \ge M - \epsilon\norm{x}_X^2. \]
		\item $\Phi$ is bounded from above on bounded sets, i.e., for every $r > 0$, there exists $K = K(r) > 0$, such that $\Phi(x) \le K$ for all $x \in B_r(0)$.
		\item $\Phi$ is Lipschitz continuous on bounded sets, i.e., for every $r > 0$, there exists $L = L(r) > 0$, such that for all $x_1, x_2 \in B_r(0)$, we have
			\[ \abs{\Phi(x_1) - \Phi(x_2)} \le L \norm{x_1 - x_2}_X. \]
	\end{enumerate}
\end{lemma}
\begin{proof}
	We verify that these properties hold.
	Property (iii) is trivially satisfied.
	Property (ii) is satisfied with $K := \max \{\Phi(0), 0\} + L(r)r$ by the Lipschitz continuity of $\Phi$ on bounded sets, as
	\[ \Phi(x) \le \Phi(0) + L\norm{x}_X \le \Phi(0) + L(r)r \le K \]
	for all $x \in B_r(0)$.
	By \cref{ass:phi} \ref{ass:phi_lower_bound} we moreover have
	\begin{equation*}
		\Phi(x) + \varepsilon\left\|x\right\|_X^2
		\ge \Phi(0) - \underline{L}\norm{x}_X + \varepsilon\norm{x}_X^2
		= \Phi(0) + \varepsilon \left(\left\|x\right\|_X - \frac{\underline{L}}{\varepsilon}\right)\norm{x}_X
	\end{equation*}
	for all $\varepsilon > 0$ and $x \in X$.
	Now, the minimum of the function $f(t) = \varepsilon(t - \frac{\underline{L}}{\varepsilon})t$ on $\R$ is attained in $\frac{\underline{L}}{2\varepsilon}$, so that for given $\varepsilon > 0$, property (i) is satisfied with
	\begin{equation*}
		M := \Phi(0) + f\left(\frac{\underline{L}}{2\varepsilon}\right) = \Phi(0) - \frac{\underline{L}^2}{4\varepsilon}.
	\end{equation*}
\end{proof}

\begin{proof}[Proof of \cref{form_OM_funct}]
	By \eqref{post_dist} and the continuity of $\Phi$, $\mu_0$ and $\mu^y$ are equivalent and thus have the same space $E = \ran(Q^{1/2})$ of admissible shifts, see \cref{OM_funct_gauss}.
	By Theorem 3.2 in \cite{Dashti:2013}, the Onsager--Machlup functional $I$ of $\mu^y$ is given by \eqref{def_I} if $\Phi$ satisfies properties (i) to (iii) of \cref{cond_Dashti_hold}, which is the case by \cref{cond_Dashti_hold}.
\end{proof}

\begin{proof}[Proof of \cref{min_I_ex}]
	By Theorem 5.4 in \cite{stuart_2010}, $I$ has a minimizer in $E$ if $\Phi$ satisfies conditions (i) and (iii) of \cref{cond_Dashti_hold}, which is the case by \cref{cond_Dashti_hold}.
\end{proof}

\begin{proof}[Proof of \cref{weak_strong_modes}]
	By Lemma 4.5 in \cite{klebanov2023}, we have
	\[ \limsup_{\delta \to 0} \frac{\mu_0(B_\delta(z))}{\mu_0(B_\delta(0))} = 0 \]
	for any $z \in X \setminus E$, because the separable Hilbert space $X$ is isomorphic to $\ell^2$. That is, $(\mu_0,E)$ has the $M$-property. 
	$\Phi$ is bounded on bounded subsets by \cref{ass:phi} \ref{ass:phi_loc_lip}.
	This allows us to apply Lemma B.1 (b) in \cite{Ayanbayev_2021a}, which yields that $(\mu^y,E)$ has the $M$-property as well.
	By Proposition 4.1 in \cite{Ayanbayev_2021a}, $x \in X$ thus is a weak mode of $\mu^y$ if and only if it minimizes its Onsager--Machlup functional $I$. We know, however, from \cref{var_char_modes} that the minimizers of $I$ are precisely the strong modes of $\mu^y$.
\end{proof}


\section{Proofs of section \ref{sect_ip}}
\label{sect_proofs_ip}

\begin{proof}[Proof of \cref{cond_gauss}]
	A straightforward computation yields
	\begin{equation}
	\label{gauss_diff}
	\begin{split}
		\Phi(u_1,y) - \Phi(u_2,y) &= \frac12\norm{Q^{-\frac12}F(u_1) - Q^{-\frac12}F(u_2)}_Y^2 \\
		&\quad+ \scalprod{Q^{-1}F(u_1) - Q^{-1}F(u_2), F(u_2) - y}_Y \\
		&= \frac12\scalprod{Q^{-1}F(u_1) - Q^{-1}F(u_2), F(u_1) - F(u_2)}_Y \\
		&\quad+ \scalprod{Q^{-1}F(u_1) - Q^{-1}F(u_2), F(u_2) - F(0)}_Y \\
		&\quad+ \scalprod{Q^{-1}F(u_1) - Q^{-1}F(u_2), F(0) - y}_Y
	\end{split}
	\end{equation}
	for all $u_1,u_2 \in X$.
	By the boundedness of $Q$, $F$ is Lipschitz continuous on bounded sets as well, so that
	\begin{align*}
		\abs{\Phi(u_1,y) - \Phi(u_2,y)} &\le \frac{L_1 L_2}{2}\norm{u_1 - u_2}_X^2 + L_1 L_2 \norm{u_1 - u_2}_X \norm{u_2}_X \\
		&\quad+ L_1 \norm{u_1 - u_2}_X \norm{F(0) - y}_Y \\
		&\le L_1\norm{u_1 - u_2}_X \left(L_2 r + L_2 r + \norm{F(0) - y}_Y\right)
	\end{align*}
	for all $u_1,u_2 \in B_r(0)$, where $L_1$ denotes the Lipschitz constant of $Q^{-1} \circ F$ on $B_r(0)$, and $L_2$ that of $F$ on $B_r(0)$.
	This proves \cref{ass:phi} \ref{ass:phi_loc_lip}.

	Setting $u_1 = u$ and $u_2 = 0$ in \eqref{gauss_diff} moreover yields
	\[ \Phi(u,y) - \Phi(0,y) = \frac12\norm{Q^{-\frac12}F(u) - Q^{-\frac12}F(0)}_Y^2 + \scalprod{Q^{-1}F(u) - Q^{-1}F(0), F(0) - y}_Y \]
	for all $u \in X$.
	By assumption, we thus have
	\[ \Phi(u,y) - \Phi(0,y) \ge 0 - \norm{Q^{-1}F(u) - Q^{-1}F(0)}_Y \norm{F(0) - y}_Y \ge -L_0 \norm{u}_X \norm{F(0) - y}_Y \]
	for all $u \in X$.
	This proves \cref{ass:phi} \ref{ass:phi_lower_bound}.
\end{proof}

\begin{proof}[Proof of \cref{cond_white_gauss_lin}]
	As a bounded linear operator on $X$, $Q^{-1}A$ is Lipschitz continuous. Therefore, it is Lipschitz continuous on bounded sets and satisfies condition \eqref{Lip0_cond}. Now it follows from \cref{cond_gauss} that $\Phi(u,y) = \Psi(Au,y)$ satisfies \cref{ass:phi}.
\end{proof}

\begin{proof}[Proof of \cref{cond_white_gauss}]
	Similarly as in the proof of \cref{cond_gauss}, we compute
	\begin{equation}
	\label{gauss_white_diff}
	\begin{split}
		\Phi(u_1,y) - \Phi(u_2,y) &= \frac{1}{2\sigma^2}\norm{F(u_1) - F(u_2)}_Y^2
		+ \frac{1}{\sigma^2}\dualpair{F(u_2) - y, F(u_1) - F(u_2)}_{V' \times V} \\
		&= \frac{1}{2\sigma^2}\norm{F(u_1) - F(u_2)}_Y^2
		+ \frac{1}{\sigma^2}\scalprod{F(u_2) - F(0), F(u_1) - F(u_2)}_Y \\
		&\quad+ \frac{1}{\sigma^2}\dualpair{F(0) - y, F(u_1) - F(u_2)}_{V' \times V}
	\end{split}
	\end{equation}
	for all $u_1,u_2 \in X$, where we identified $F(u_2) \in V$ with $\scalprod{F(u_2),\cdot}_Y \in V'$.
	By assumption, $Q^{-1/2} \circ F$ is Lipschitz continuous on bounded sets.
	Therefore, $F$ is Lipschitz contiuous on bounded sets as well by the boundedness of $Q^{1/2}$. We obtain that
	\begin{align*}
		\abs{\Phi(u_1,y) - \Phi(u_2,y)} &\le \frac{L_2^2}{2\sigma^2}\norm{u_1 - u_2}_X^2 + \frac{L_2^2}{\sigma^2} \norm{u_2}_X \norm{u_1 - u_2}_X
		+ \frac{1}{\sigma^2} \norm{F(0) - y}_{V'} \norm{u_1 - u_2}_V \\
		&\le \left(L_2^2 r + L_2^2 r + L_1 \norm{F(0) - y}_{V'}\right) \norm{u_1 - u_2}_X
	\end{align*}
	for all $u_1,u_2 \in B_r(0)$, where $L_1$ denotes the Lipschitz constant of $Q^{-1/2} \circ F$ on $B_r(0)$, and $L_2$ that of $F$ on $B_r(0)$.
	This proves \cref{ass:phi} \ref{ass:phi_loc_lip}.

	Setting $u_1 = u$ and $u_2 = 0$ in \eqref{gauss_white_diff} moreover yields
	\[ \Phi(u,y) - \Phi(0,y) = \frac{1}{2\sigma^2}\norm{F(u) - F(0)}_Y^2 + \dualpair{F(0) - y, F(u) - F(0)}_{V' \times V} \]
	for all $u \in X$.
	By assumption, we thus have
	\[ \Phi(u,y) - \Phi(0,y) \ge 0 - \norm{F(0) - y}_{V'} \norm{F(u) - F(0)}_V \ge -L_0 \norm{F(0) - y}_{V'} \norm{u}_X \]
	for all $u \in X$.
	This proves \cref{ass:phi} \ref{ass:phi_lower_bound}.
\end{proof}

\begin{proof}[Proof of \cref{cond_white_gauss_lin}]
	By the assumptions on $A$, $Q := AA^*$ is a self-adjoint, positive, trace class operator on $Y$ with $\ran(Q^{1/2}) = \ran((AA^*)^{1/2}) = \ran(A)$. Moreover, $Q^{-1/2}A$ is bounded since 
	\[ \norm{Q^{-1/2}Ax}_Y = \norm{(AA^*)^{-1/2}Ax}_Y = \norm{x}_X \quad \text{for all}~x \in X. \]
	Now it follows from \cref{cond_white_gauss} that $\Phi$ satisfies \cref{ass:phi}, cf.~\cref{rem:cond_white_gauss_lin}.
\end{proof}

\begin{proof}[Proof of \cref{xi_well_def}]
We prove that the series $\sum_{k=1}^\infty \abs{\xi_k}^2$
converges almost surely using Kolmogorov's three series theorem with $A > 0$, $X_k := \abs{\xi_k}^2$ and $Y_k := X_n\mathbf{1}_{\abs{X_n} \le A^2}$ for all $k \in \N$.
Without loss of generality, we may consider the case $a = 0$, since convergence of $\sum_{k=1}^\infty \abs{\xi_k}^2$ with $\xi_k \sim \Laplace_{\lambda_k}$ for all $k \in \N$ implies convergence of $\sum_{k=1}^\infty \abs{\tilde{\xi}_k}^2 = \norm{a}_X^2 + \sum_{k=1}^\infty \abs{\xi_k}^2$ with $\tilde{\xi}_k \sim \Laplace_{\scalprod{a,e_k}_X,\lambda_k}$ for all $k \in \N$.

First of all, we have
\begin{equation*}
	\Prob{\abs{X_k} > A^2} = \Prob{\abs{\xi_k}^2 > A^2} = 2 \Prob{\xi > A} \\
	= \frac{2}{\sqrt{2\lambda_k}} \int_A^\infty e^{-\frac{\sqrt{2}x}{\sqrt{\lambda_k}}} \di x
	= e^{-\frac{\sqrt{2}A}{\sqrt{\lambda_k}}}.
\end{equation*}
Since $t^2e^{-t} \to 0$ as $t \to \infty$ and $\lambda_k \to 0$ as $k \to \infty$, this implies
\[ \sum_{k=1}^\infty \Prob{\abs{X_k} > A^2} \le C_1 \sum_{k=1}^\infty \lambda_k = C_1 \Tr Q < \infty. \]

Second of all, we have
\[ 0 \le \sum_{k=1}^\infty \Exp{Y_k} \le \sum_{k=1}^\infty \Exp{X_k} = \sum_{k=1}^\infty \Exp{\abs{\xi_k}^2} = \sum_{k=1}^\infty \lambda_k = \Tr Q < \infty. \]

Last of all, we have
\[ \Exp{\abs{\xi_k}^4} = \sqrt{\frac{2}{\lambda_k}} \int_0^\infty x^4 e^{-\sqrt{\frac{2}{\lambda_k}}x} \di x
	= \frac{\lambda_k^2}{4} \int_0^\infty t^4 e^{-t} \di t = \frac{\lambda_k^2}{4} 4! = 6\lambda_k^2. \]
This yields
\begin{align*}
	\Var(Y_k) &\le \Var(X_k) = \Exp{\left(\abs{\xi_k}^2 - \Exp{\abs{\xi_k}^2}\right)^2} 
	= \Exp{\abs{\xi_k}^4 - 2\abs{\xi_k}^2\Exp{\abs{\xi_k}^2} + \Exp{\abs{\xi_k}^2}^2} \\
	&= \Exp{\abs{\xi_k}^4} - \Exp{\abs{\xi_k}^2}^2 = 6\lambda_k^2 - \lambda_k^2 = 5\lambda_k^2,
\end{align*}
and in turn
\[ \sum_{k=1}^\infty \Var{Y_k} \le 5 \sum_{k=1}^\infty \lambda_k^2 \le C_2 \sum_{k=1}^\infty \lambda_k = C_2 \Tr Q < \infty. \]

Now, $\norm{\xi}_X^2 = \norm{(\xi_k)_{k\in\N}}_{\ell^2}^2 = \sum_{k=1}^\infty \abs{\xi_k}^2$ is finite almost surely by Kolmogorov's three series theorem \cite[Thm.~15.51]{Klenke:2020}.

The random sequence $(\xi_k)_{k\in\N}$ is measurable with respect to the Borel $\sigma$-algebra $\Borel(\R^\infty)$ induced by the product topology on $\R^\infty$, see section 3.4 in \cite{Kretschmann:2019}, and almost surely takes values in $\ell^2$ according to the previous considerations. By Proposition 3.5 in \cite{Kretschmann:2019}, it is therefore $\Borel(\ell^2)$-measurable. Now, $\xi$ is $\Borel(X)$-measurable by Theorems 1.80 and 1.88 in \cite{Klenke:2020} as the composition of a measurable mapping and the isomorphism $(x_k)_{k\in\N} \mapsto \sum_{k=1}^\infty x_k e_k$ between $\ell^2$ and $X$.
\end{proof}

\begin{proof}[Proof of \cref{cond_nonlin_op}]
By the triangle inequality, we have
\begin{align*}
	\abs{\Psi(F(x_1),z) - \Psi(F(x_2),z)}
	&\le \sqrt{2} \sum_{k=1}^\infty \lambda_k^{-\frac12} \Big\lvert\abs{\scalprod{y,e_k}_Y - \scalprod{F(x_1),e_k}_Y} - \abs{\scalprod{y,e_k}_Y - \scalprod{F(x_2),e_k}_Y}\Big\rvert \\
	&\le \sqrt{2} \sum_{k=1}^\infty \lambda_k^{-\frac12} \abs{\scalprod{F(x_1) - F(x_2), e_k}_Y}
\end{align*}
for all $x_1, x_2 \in X$. If condition (i) holds, then the proof is finished.
If, on the other hand, condition (ii) holds, we estimate
\begin{align*}
	\abs{\scalprod{F(x_1) - F(x_2), e_k}_Y} &= \bigabs{\bigscalprod{\int_0^1 F'(x_2 + t(x_1 - x_2))(x_1 - x_2) \di t, e_k}_X} \\
	&\le \int_0^1 \abs{\scalprod{x_1 - x_2, F'(x_2 + t(x_1 - x_2))^*e_k}_X} \di t \\
	&\le \norm{x_1 - x_2}_X \int_0^1 \norm{F'(x_2 + t(x_1 - x_2))^*e_k}_X \di t
\end{align*}
for all $x_1, x_2 \in X$ using the Fréchet differentiability of $F$.
Now, condition \eqref{cond_F} implies that
\begin{align*}
	\abs{\Psi(F(x_1),z) - \Psi(F(x_2),z)}
	&\le \bigg(\sqrt{2} \int_0^1 \sum_{k=1}^\infty \lambda_k^{-\frac12} \norm{F'(x_2 + t(x_1 - x_2))^*e_k}_X \di t\bigg) \norm{x_1 - x_2}_X \\
	&\le \sqrt{2}C\norm{x_1 - x_2}_X
\end{align*}
for all $x_1, x_2 \in X$.
\end{proof}

\begin{proof}[Proof of \cref{cond_lin_op}]
Let $(\kappa_k,f_k,e_k)_{k\in\N}$ denote a singular system of $K$.
By the triangle inequality and the Cauchy--Schwarz inequality, we have
\begin{align*}
	\abs{\Psi(Kx_1,z) - \Psi(Kx_2,z)} &\le \sqrt{2} \sum_{k=1}^\infty \lambda_k^{-\frac12} \Big\lvert\abs{\scalprod{y,e_k}_Y - \scalprod{Kx_1,e_k}_Y} - \abs{\scalprod{y,e_k}_Y - \scalprod{Kx_2,e_k}_Y}\Big\rvert \\
	&\le \sqrt{2} \sum_{k=1}^\infty \lambda_k^{-\frac12} \abs{\scalprod{x_1 - x_2,K^* e_k}_X}
	= \sqrt{2} \sum_{k=1}^\infty \frac{\kappa_k}{\lambda_k^{1/2}} \abs{\scalprod{x_1 - x_2,f_k}_X} \\
	&\le \sqrt{2} \left(\sum_{k=1}^\infty \frac{\kappa_k^2}{\lambda_k}\right)^\frac12 \left(\sum_{k=1}^\infty \abs{\scalprod{x_1 - x_2,f_k}_X}^2\right)^\frac12
\end{align*}
for all $x_1,x_2 \in X$.
Since $(f_k)_{k\in\N}$ is an orthonormal system in $X$ and $(e_k)_{k\in\N}$ an orthonormal basis of $Y$, it follows that
\begin{align*}
	\abs{\Psi(Kx_1,z) - \Psi(Kx_2,z)} &\le \sqrt{2} \left(\sum_{k=1}^\infty \bignorm{K^*Q^{-\frac12}e_k}_X^2\right)^\frac12 \norm{x_1 - x_2}_X \\
	&= \sqrt{2} \bignorm{K^*Q^{-\frac12}}_\text{HS} \norm{x_1 - x_2}_X
	= \sqrt{2} \bignorm{Q^{-\frac12}K}_\text{HS} \norm{x_1 - x_2}_X
\end{align*}
for all $x_1,x_2 \in X$.
\end{proof}


\section{Proofs of section \ref{sect_lin_prob}}
\label{sect_proofs_lin_prob}

\begin{proof}[Proof of \cref{lemSolEst}]
	First, we consider
	\begin{equation*}
		\sum_{k=1}^\infty \alpha_k^s \abs{\scalprod{e^{-tA}u, \varphi_k}_X}^2
		= t^{-s} \sum_{k=1}^\infty (\alpha_k t)^s e^{-2\alpha_k t} \abs{\scalprod{u,\varphi_k}_X}^2.
	\end{equation*}
	By \cite[Lem.~5.14]{Kretschmann:2019}, the sequence $(\alpha_k t)^s e^{-2\alpha_k t}$ is bounded from above by $C := s^s e^{-s}$, so that
	\begin{equation*}
		\sum_{k=1}^\infty \alpha_k^s \abs{\scalprod{e^{-tA}u, \varphi_k}_X}^2
		\le t^{-s} C \norm{u}_X^2 < \infty.
	\end{equation*}
	This implies $e^{-tA}u \in \X{s}$ and proves the estimate.
\end{proof}

\begin{proof}[Proof of \cref{prop:phi_cont}]
	We show that $\Phi(\cdot,y)$ is Lipschitz continuous by applying \cref{cond_lin_op} with $K = e^{-A}$ and $Q = b^2A^{-\beta}$.
	Here, we have
	\[ \bignorm{K^*Q^{-\frac12}}_\text{HS}^2 = \sum_{k=1}^\infty \bignorm{K^*Q^{-\frac12}e_k}_X^2 = \frac{1}{b^2} \sum_{k=1}^\infty e^{-2\alpha_k} \alpha_k^{\beta} = \frac{1}{b^2} \sum_{k=1}^\infty \left(e^{-\alpha_k} \alpha_k^{\beta}\right)^2 \alpha_k^{-\beta}. \]
	The sequence $(\alpha_k^\beta e^{-\alpha_k})_{k\in\N}$ is bounded from above by $\beta^\beta e^{-\beta}$, see \cite[Lem.~5.14]{Kretschmann:2019}, and the operator $A^{-\beta}$ is trace class according to \eqref{eq:pow_A_trace_class}.
	This leads to the estimate
	\[ \bignorm{K^*Q^{-\frac12}}_\text{HS}^2 \le \frac{1}{b^2} \left(\beta^\beta e^{-\beta}\right)^2 \Tr A^{-\beta}, \]
	which shows that the assumptions of \cref{cond_lin_op} are satisfied.

	Now we show the continuity in $y$.
	Let $u \in X$ and $\varepsilon > 0$. Here, we estimate
	\begin{align*}
		\abs{\Phi(u,y) - \Phi(u,z)}
		&= \left| \frac{\sqrt{2}}{b} \sum_{k=1}^\infty \alpha_k^\frac{\beta}{2} \left( \abs{y_k - e^{-\alpha_k}u_k} - \abs{y_k} - \abs{z_k - e^{-\alpha_k}u_k} + \abs{z_k} \right) \right| \\
		&\le \frac{\sqrt{2}}{b} \sum_{k=1}^N 2\alpha_k^\frac{\beta}{2} \abs{y_k - z_k} 
		+ \frac{\sqrt{2}}{b} \sum_{k=N+1}^\infty 2\alpha_k^\frac{\beta}{2} \abs{e^{-\alpha_k}u_k}
	\end{align*}
	for all $y,z \in X$ and $N \in \N$.
	As the series $\sum_{k=1}^\infty \alpha_k^\frac{\beta}{2} e^{-\alpha_k} \abs{u_k}$ converges, we can choose $N = N(u)$ such that
		\[ \frac{\sqrt{2}}{b} \sum_{k=N+1}^\infty 2\alpha_k^\frac{\beta}{2}e^{-\alpha_k}\abs{u_k} \le \frac{\varepsilon}{2}. \]
	Next, we choose
		\[ \delta \coloneqq \frac{b}{2\sqrt{2}}\left(\sum_{k=1}^N \alpha_k^{\beta}\right)^{-\frac12}\frac{\varepsilon}{2}. \]
	This way, we have
	\begin{equation*}
		\frac{\sqrt{2}}{b} \sum_{k=1}^N 2\alpha_k^\frac{\beta}{2} \abs{y_k - z_k}
		\le \frac{2\sqrt{2}}{b} \left(\sum_{k=1}^N \alpha_k^{\beta}\right)^\frac12 
		\left(\sum_{k=1}^N \abs{y_k - z_k}^2\right)^\frac12
		\le \frac{\varepsilon}{2\delta} \norm{y - z}_{X}
		\le \frac{\varepsilon}{2}
	\end{equation*}
	for all $y,z \in X$ with $\norm{y - z}_{X} \le \delta$, and consequently
		\[ \left|\Phi(u,y) - \Phi(u,z)\right| \le \frac{\varepsilon}{2} + \frac{\varepsilon}{2} = \varepsilon. \]
	The continuity of $\Phi$ now follows from the continuity in $u$ and $y$ and the triangle inequality.
\end{proof}

\begin{lemma}
	\label{lemExpNormInt}
	Let $\mu_0$ be a centered Gaussian measure on a separable Hilbert space $X$.
	Then for every $C > 0$, the function $u \mapsto \exp(C\norm{u}_X)$ defined on $X$ is $\mu_0$-integrable.
\end{lemma}
\begin{proof}
	By Fernique's theorem \cite[Thm~2.8.5]{Bogachev:1998}, there exists $\alpha > 0$ such that the integral $\int_X \exp(\alpha\norm{u}_X^2) \mu_0(\di u)$ is finite. Set $R \coloneqq \frac{C}{\alpha}$. Then the integral
	\begin{equation*}
		\int_X \exp\left(C\norm{u}_X\right) \mu_0(\di u)
		\le \int_{B_R(0)} \exp\left(\frac{C^2}{\alpha}\right) \mu_0(\di u)
		+ \int_{X \setminus B_R(0)} \exp\left(\alpha\norm{u}_X^2\right) \mu_0(\di u)
	\end{equation*}
	is finite as well.
\end{proof}

\begin{proof}[Proof of \cref{prop:Z_pos_fin}]
	We first show the integrability.
	Let $y \in X$ be arbitrary.
	We use the Lipschitz continuity of $\Phi(\cdot,y)$, which holds by Proposition \ref{prop:phi_cont}, to estimate
	\begin{equation*}
		\int_X \exp(-\Phi(u,y)) \Normal_{r^2 A^{-\tau}}(\di u)
		\le \exp(-\Phi(0,y)) \int_X \exp(L\norm{u}_X) \Normal_{r^2 A^{-\tau}}(\di u).
	\end{equation*}
	Now $\Phi(0,y) = 0$ for all $y \in X$ by definition of $\Phi$ and the integral on the right hand side is finite by Lemma \ref{lemExpNormInt}.

	Now, we address the lower bound.
	By the Lipschitz continuity of $\Phi$ in $u$, the estimate
	\begin{align*}
		\int_X \exp(-\Phi(u,y)) \Normal_{r^2 A^{-\tau}}(\di u)
		&\ge \int_X \exp(-L\norm{u}_X) \Normal_{r^2 A^{-\tau}}(\di u) \\
		&\ge \int_{B_1(0)} e^{-L} \Normal_{r^2 A^{-\tau}}(\di u)
		= e^{-L} \Normal_{r^2 A^{-\tau}}(B_1(0)) =: C_Z
	\end{align*}
	holds for all $y \in X$.
	By Theorem 3.6.1 in \cite{Bogachev:1998}, the topological support of the Gaussian measure $\Normal_{r^2 A^{-\tau}}$ is given by the closure of its Cameron--Martin space $\ran(A^{-\tau/2})$. In particular, all balls around $0$ have positive measure under $\Normal_{r^2 A^{-\tau}}$ because $0 \in \ran(A^{-\tau/2})$, which implies that the constant $C_Z$ is positive.
\end{proof}


\bibliography{references}

\end{document}